\documentclass[english]{amsart}
\usepackage[T1]{fontenc}
\usepackage[latin9]{inputenc}
\usepackage{color}
\usepackage{amsthm}
\usepackage{setspace}
\usepackage{amssymb}
\makeatletter
\numberwithin{equation}{section}
\numberwithin{figure}{section}
\theoremstyle{plain}
\newtheorem{thm}{Theorem}[section]
 \theoremstyle{definition}
  \newtheorem{example}[thm]{Example}
  \theoremstyle{definition}
  \newtheorem{defn}[thm]{Definition}
  \theoremstyle{plain}
  \newtheorem{lem}[thm]{Lemma}
  \theoremstyle{remark}
  \newtheorem{rem}[thm]{Remark}

\makeatother

\usepackage{babel}

\begin{document}

\title{Rees Algebras of Diagonal Ideals}

\author{Kuei-Nuan Lin}

\subjclass[2000]{\textcolor{black}{Primary 13C40, 14M12 ; Secondary 13P10, 14Q15. }}

\keywords{\textcolor{black}{Rees Algebra, Join Variety, Determinantal Ring,
Gröbner Basis, Symmetric Algebra.}}
\begin{abstract}
There is a natural epimorphism from the symmetric algebra to the Rees
algebra of an ideal. When this epimorphism is an isomorphism, we say
that the ideal is of linear type. Given two determinantal rings over
a field, we consider the diagonal ideal, the kernel of the multiplication
map. We prove that the diagonal ideal is of linear type and recover
the defining ideal of the Rees algebra in some special cases. The
special fiber ring of the diagonal ideal is the homogeneous coordinate
ring of the join variety. 
\end{abstract}
\maketitle

\section{INTRODUCTION}

In this paper we address the problem of determining the equations
that define the Rees algebra of an ideal. Besides encoding asymptotic
properties of the powers of an ideal, the Rees algebra realizes, algebraically,
the blow-up of a variety along a subvariety. Though blowing up is
a fundamental operation in the birational study of algebraic varieties
and, in particular, in the process of desingularization, an explicit
description of the resulting variety in terms of defining equations
remains a difficult problem. 

Let $I$ be an ideal in a Noetherian ring $R$. The Rees algebra $\mathcal{R}(I)$
of $I$ is the graded subalgebra $R[It]\cong\oplus_{n\geq0}I^{n}$
of $R[t]$. When $I$ is generated by $f_{1},...,f_{u}$, there is
a natural map $\phi$ from $R[t_{1},...,t_{u}]$ to $\mathcal{R}(I)$
sending $t_{i}$ to $f_{i}t$. The kernel of $\phi$ is the defining
ideal of $\mathcal{R}(I)$ in the ring $R[t_{1},...,t_{u}]$. There
is another natural map $\psi$ from $\mbox{Sym}(R^{u})=R[t_{1},...,t_{u}]$
to $\mbox{Sym}(I)$, the symmetric algebra of $I$, and the kernel
of $\psi$ is the defining ideal of $\mbox{Sym}(I)$. This ideal is
generated by the entries of the product of $(t_{1},...,t_{u})$ and
the presentation matrix of $I$. The defining ideal of $\mbox{Sym}(I)$
is contained in the kernel of $\phi$. Hence we have a surjective
map from $\mbox{Sym}(I)$ to $\mathcal{R}(I)$. The ideal $I$ is
said to be of \textit{linear type} if $\mbox{Sym}(I)$ is naturally
isomorphic to $\mathcal{R}(I)$. Hence we obtain the defining equations
of $\mathcal{R}(I)$ for free in this case. 

In general, an ideal is not of linear type. The first known class
of ideals of linear type are complete intersection ideals \cite{M}.
Ideals generated by $d$-sequences are another large class of ideals
of linear type \cite{H1}, \cite{V}. These sequences play a role
in the theory of approximation complexes similar to the role regular
sequences play in the theory of Koszul complexes. Later Herzog-Simis-Vansconcelos
and Herzog-Vansconcelos-Villarreal used strongly Cohen-Macaulay and
sliding depth conditions to describe classes of ideals of linear type
\cite{HSV1}, \cite{HSV2}, \cite{HVV}. Huneke proved that if $X$
is a generic $n\times n$ matrix and $I$ is the ideal of $n-1$ size
minors of $X$ in $R=\mathbb{Z}[x_{ij}]$, then $I$ is of linear
type \cite{H2}. Villarreal showed the edge ideals of a tree or a
graph with a unique odd cycle are ideals of linear type \cite{Vi}.
In this paper, we give a new class of ideals of linear type, diagonal
ideals of determinantal rings.

Now we describe the setting of this work. Let $k$ be a field, $R$
be a polynomial ring over the field $k$ with variables $\{x_{ij}\}$,
and $X$ is the generic $m\times n$ matrix $(x_{ij})$. Given two
homogeneous $R$-ideals, $I_{1}$ and $I_{2}$, we consider the kernel
of the multiplication map from $S=R/I_{1}\otimes R/I_{2}$ to $R/(I_{1}+I_{2})$.
The kernel is the \textit{diagonal} ideal $\mathbb{D}$ of the ring
$S$ and $ $$\mathbb{D}$ is generated by the images of $x_{i}\otimes1-1\otimes x_{i}$
in the ring $S$. The main result of this paper shows that the ideal
$\mathbb{D}$ is of linear type if $I_{1}$, $I_{2}$ are the ideals
of maximal minors of submatrics of $X$. Notice $I_{1}$ and $I_{2}$
are in general not of linear type (see \cite{H2}, 2.6). 

In this particular case, the \textit{special fiber ring} of $I$,
$\mathcal{F}(\mathbb{D})=\mathcal{R}(\mathbb{D})\otimes_{S}k$, is
the homogeneous coordinate ring of the embedded join varieties of
$V(I_{1})$ and $V(I_{2})$ in projective space $\mathbb{P}_{k}^{m\times n-1}$.
Hence when $\mathbb{D}$ is an ideal of linear type, the embedded
join is the whole space. But it is not true in general that if the
embedded join variety is the whole space, the diagonal ideal $\mathbb{D}$
is of linear type. See Example \ref{notFiber} in Section 2. 

The proof of the main result is in Section 2. We now describe the\textcolor{black}{{}
idea of the proof. We use the defining ideals of $\mbox{Sym}(\mathbb{D})$
to understand the defining ideals of $\mathcal{R}(\mathbb{D})$. }We
identify some specific equations in the defining ideal $\mathcal{J}$
of $\mbox{Sym}(\mathbb{D})$, and consider the subideal $\mathcal{L}$
of $\mathcal{J}$ they generate.

\textcolor{black}{Notice that $\mathcal{L}\subset\mathcal{J}\subset\mathcal{K}$,
where $\mathcal{K}$ is the defining ideal of $\mathcal{R}(\mathbb{D})$,
hence the goal is to prove that $\mathcal{L}=\mathcal{K}$. We use
Buchberger's Algorithm to find a Gröbner basis of the ideal $\mathcal{L}$
with respect to some monomial order. }We find a set of polynomials
that are in the ideal $\mathcal{L}$ and show that all the remainders
between elements in this set are zero. Hence we find a Gröbner basis
of the ideal $\mathcal{L}$. Once we have the Gröbner basis, we have
the generating set for the initial ideal $\mbox{in}(\mathcal{L})$
of $\mathcal{L}$. This way we find a non zero-divisor modulo \textcolor{black}{$\mathcal{L}$},
which we may invert thereby reducing to the case of a smaller matrix.
Thus we show that $\mathcal{L}=\mathcal{K}$. As a consequence, the
two algebras $\mbox{Sym}(\mathbb{D})$ and $\mathcal{R}(\mathbb{D})$
are naturally isomorphic and we obtain an explicit description of
the defining equations of $\mathcal{R}(\mathbb{D})$.

\medskip{}

\begin{flushleft}
\textbf{Acknowledgments:} This work is based on author's Ph. D. thesis
from Purdue University under the direction of Professor Bernd Ulrich.
The author is very grateful for so many useful suggestions from Professor
Ulrich.
\par\end{flushleft}

\section{Main results}

Let $k$ be a field, $2\leq m\leq n$ integers, \textcolor{black}{$X_{mn}=[x_{ij}],\: Y_{mn}=[y_{ij}],$}
$Z_{mn}=[z_{ij}]$, $m\times n$ matrices of variables over $k$.
Let\textcolor{black}{{} $2\leq s_{i}\leq t_{i}$ integers, and let $X_{s_{1}t_{1}}$,
$Y_{s_{2}t_{2}}$ be the submatrices of $X$ and $Y$ consisting of
the first $s_{i}$ rows and first $t_{i}$ columns.} We write $I=I_{s_{1}}(X_{s_{1}t_{1}})$,
$J=I_{s_{2}}(X_{s_{2}t_{2}})$ the ideals of $k[X]$ generated by
the maximal minors of $X_{s_{1}t_{1}}$ and the maximal minors of
$X_{s_{2}t_{2}}$. Let $R_{1}=k[X]/I$, $R_{2}=k[X]/J$ be the two
determinantal rings. We consider the diagonal ideal $\mathbb{D}$
of $R_{1}\otimes_{k}R_{2}$, defined via the exact sequence \[
0\longrightarrow\mathbb{D}\longrightarrow R_{1}\otimes_{k}R_{2}\overset{_{\mathrm{mult.}}}{\longrightarrow}k[X]/(I+J)\longrightarrow0.\]
 The ideal $\mathbb{D}$ is generated by the images of $x_{ij}\otimes1-1\otimes x_{ij}$
in $R_{1}\otimes_{k}R_{2}$. 

\textcolor{black}{We write the diagonal ideal $\mathbb{D}=(\{x_{ij}-y_{ij}\})$
in \[
S=k[X_{mn},Y_{mn}]/(I_{s_{1}}(X_{s_{1}t_{1}}),I_{s_{2}}(Y_{s_{2}t_{2}}))\cong R_{1}\otimes_{k}R_{2}.\]
We have a presentation of $\mathbb{D}$,\[
\begin{array}{cccccc}
S^{l} & \overset{\phi}{\longrightarrow} & S^{mn} & \longrightarrow\mathbb{D} & \longrightarrow & 0\end{array}\]
From this we obtain a presentation of the symmetric algebra of $\mathbb{D}$,
\[
0\rightarrow(\mbox{image}(\phi))=J\longrightarrow\mbox{Sym}(S^{mn})=S[Z_{mn}]=T\longrightarrow\mbox{Sym}(\mathbb{D})\rightarrow0.\]
Here $J$ is the ideal generated by the entries of the row vector
$[z_{11},z_{12},...,z_{1n},....,z_{mn}]\cdot\phi$.} Hence \textcolor{black}{\[
\mbox{Sym}(\mathbb{D})\cong T/J\]
}where $J$ is generated by linear forms in the variables $z_{ij}$.
We write $\mathcal{R}(\mathbb{D})=T/K$, $J\subset K$. In general
$K$ is not generated by linear forms. We can rewrite $\mbox{Sym}(\mathbb{D})=T/J=k[X_{mn},Y_{mn,}Z_{mn}]/\mathcal{J}$
and $\mathcal{R}(\mathbb{D})=k[X_{mn},Y_{mn},Z_{mn}]/\mathcal{K}$.
In this particular case, the \textit{special fiber ring} of $I$,
$\mathcal{F}(\mathbb{D})=\mathcal{R}(\mathbb{D})\otimes_{S}k$, is
the homogeneous coordinate ring of the embedded join varieties of
$V(I_{1})$ and $V(I_{2})$ in projective space $\mathbb{P}_{k}^{m\times n-1}$. 
\begin{thm}
\textcolor{black}{\label{linear} The ideal $\mathbb{D}$ is of linear
type if $I_{1}$ and $I_{2}$ are generated by the maximal minors
of submatrices of $X$. So $\mathcal{R}(\mathbb{D})\cong\mathrm{Sym}(\mathbb{D})$.}
\end{thm}
Hence the embedded join is the whole space in the above case. \textcolor{black}{The
following example showing that even the fiber ring is the whole space,
the ideal $\mathbb{D}$ may not be of linear type in general.}
\begin{example}
\label{notFiber}Let $X$, $Y$, and $Z$ be $3\times3$ matrices
and $I_{1}=I_{3}(X)$, $I_{2}=I_{2}(X)$ be the ideal generated by
$3\times3$ and $2\times2$ minors of $X$. Write \[
S=k[X,Y]/(I_{3}(X),I_{2}(Y))\cong R_{1}\otimes_{k}R_{2},\]
 $\mbox{Sym}(\mathbb{D})=S[Z]/J$ and $\mathcal{R}(\mathbb{D})=S[Z]/K$.
Then $J=(g_{ij,lk}$, $f)$ where $g_{ij,lk}=(\{(x_{ij}-y_{ij})z_{lk}-(x_{lk}-y_{lk})z_{ij}\})$
and \[
f=\left|\begin{array}{ccc}
x_{11} & x_{12} & x_{13}\\
z_{21} & z_{22} & z_{23}\\
y_{31} & y_{32} & y_{33}\end{array}\right|+\left|\begin{array}{ccc}
x_{11} & x_{12} & x_{13}\\
x_{21} & x_{22} & x_{23}\\
z_{31} & z_{32} & z_{33}\end{array}\right|.\]
 $K=(J,$ $h)$ where \[
h=\left|\begin{array}{ccc}
z_{11} & z_{12} & z_{13}\\
z_{21} & z_{22} & z_{23}\\
y_{31} & y_{32} & y_{33}\end{array}\right|+\left|\begin{array}{ccc}
z_{11} & z_{12} & z_{13}\\
y_{21} & y_{22} & y_{23}\\
z_{31} & z_{32} & z_{33}\end{array}\right|+\left|\begin{array}{ccc}
x_{11} & x_{12} & x_{13}\\
z_{21} & z_{22} & z_{23}\\
z_{31} & z_{32} & z_{33}\end{array}\right|.\]

\end{example}
The remaining part of this section is devoted to prove Theorem \ref{linear}.
In the course of this we also describe the defining equations of $\mathcal{R}(\mathbb{D})$.
We identify some specific equations in the defining ideal $\mathcal{J}$
of $\mbox{Sym}(\mathbb{D})$. In order to clarify the notations, we
define matrices that will be used repeatedly.
\begin{defn}
\label{SubMatrix}Let $X=[x_{ij}]$, $Y=\left[y_{ij}\right]$ , $1\leq i\leq m,1\leq j\leq n$,
be $m$ by $n$ matrices, and $X_{a_{1}...a_{s}}^{l,k}=[x_{ia_{i}}]$,
$Y_{a_{1}...a_{s}}^{l,k}=[y_{ia_{i}}],$ $l\leq i\leq k,\,1\leq a_{i}\leq n$,
$X_{1...\hat{s}...n}^{l,k}=[x_{ij}]$, $l\leq i\leq k,\,1\leq j\leq n,j\neq s$
be submatrices. For the convenience of notation, we write $\mathrm{\mbox{det }}M=\left|M\right|$
when $M$ is a square matrix. We set determinate of a 0 by 0 matrix
equal to 1. We also write \[
\left[\begin{array}{c}
X^{1,j}\\
Y^{j+1,m}\end{array}\right]_{a_{1}...a_{m}}=\left[\begin{array}{ccc}
x_{1a_{1}} & ... & x_{1a_{m}}\\
\vdots &  & \vdots\\
x_{ja_{1}} & ... & x_{ja_{m}}\\
y_{j+1a_{1}} & ... & y_{j+1a_{m}}\\
\vdots &  & \vdots\\
y_{ma_{1}} & ... & y_{ma_{m}}\end{array}\right]\]
as a matrix with different variables.
\end{defn}
The following is a well-known fact: writing a matrix with variable
$y$'s as a matrix of variables $x$'s and a combination of differences
of $x$'s and $y$'s.
\begin{lem}
\label{YtoX}Let $X$ and $Y$ be $n\times n$ matrices. With notation
as above \[
|Y|=|X|+{\displaystyle \sum_{i=1}^{n}\sum_{j=1}^{n}(-1)^{i+j}\left|\begin{array}{c}
Y_{1...\hat{j}...n}^{1,i-1}\\
X_{1...\hat{j}...n}^{i+1,n}\end{array}\right|(y_{ij}-x_{ij})}\]
\end{lem}
\begin{proof}
We prove this by inducting on $n$. When $n=1$, the claim is trivial.
By induction, we have \[
|Y_{1...\hat{k}...n}^{2,n}|=|X_{1...\hat{k}...n}^{2,n}|+{\displaystyle \sum_{i=2}^{n}\sum_{j=1}^{k-1}(-1)^{i+j+1}\left|\begin{array}{c}
Y_{1..\hat{j}....\hat{k}....n}^{2,i-1}\\
X_{1...\hat{j}...\hat{k}...n}^{i+1,n}\end{array}\right|(y_{ij}-x_{ij})}\]
 \begin{eqnarray*}
 &  & +\sum_{i=2}^{n}\sum_{j=k+1}^{n}(-1)^{i+j}\left|\begin{array}{c}
Y_{1..\hat{k}....\hat{j}....n}^{2,i-1}\\
X_{1...\hat{k}...\hat{j}...n}^{i+1,n}\end{array}\right|(y_{ij}-x_{ij}).\end{eqnarray*}
 Hence \[
\begin{array}{ccc}
|Y| & = & \sum_{k=1}^{n}(-1)^{k+1}y_{1k}|Y_{1...\hat{k}...n}^{2,n}|\\
 & = & \sum_{k=1}^{n}(-1)^{k+1}y_{1k}\left(|X_{1...\hat{k}...n}^{2,n}|+{\displaystyle \sum_{i=2}^{n}\sum_{j=1}^{k-1}(-1)^{i+j+1}\left|\begin{array}{c}
Y_{1..\hat{j}....\hat{k}....n}^{2,i-1}\\
X_{1...\hat{j}...\hat{k}...n}^{i+1,n}\end{array}\right|(y_{ij}-x_{ij})}\right.\\
 &  & \left.+\sum_{i=2}^{n}\sum_{j=k+1}^{n}(-1)^{i+j}\left|\begin{array}{c}
Y_{1..\hat{k}....\hat{j}....n}^{2,i-1}\\
X_{1...\hat{k}...\hat{j}...n}^{i+1,n}\end{array}\right|(y_{ij}-x_{ij})\right).\end{array}\]

\[
\begin{array}{ccc}
 & {(1)\atop =} & \sum_{k=1}^{n}(-1)^{k+1}x_{1k}|X_{1...\hat{k}...n}^{2,n}|+\sum_{k=1}^{n}(-1)^{k+1}(y_{1k}-x_{1k})|X_{1...\hat{k}...n}^{2,n}|\\
 &  & +\sum_{k=1}^{n}(-1)^{k+1}y_{1k}\left(\sum_{i=2}^{n}\sum_{j=1}^{k-1}(-1)^{i+j+1}\left|\begin{array}{c}
Y_{1..\hat{j}....\hat{k}....n}^{2,i-1}\\
X_{1...\hat{j}...\hat{k}...n}^{i+1,n}\end{array}\right|(y_{ij}-x_{ij})\right.\\
 &  & \left.+\sum_{i=2}^{n}\sum_{j=k+1}^{n}(-1)^{i+j}\left|\begin{array}{c}
Y_{1..\hat{k}....\hat{j}....n}^{2,i-1}\\
X_{1...\hat{k}...\hat{j}...n}^{i+1,n}\end{array}\right|(y_{ij}-x_{ij})\right)\end{array}\]
\begin{eqnarray*}
 & = & |X|+\sum_{k=1}^{n}(-1)^{k+1}(y_{1k}-x_{1k})|X_{1...\hat{k}...n}^{2,n}|\\
 &  & +\sum_{k=1}^{n}(-1)^{k+1}y_{1k}\left(\sum_{i=2}^{n}\sum_{j=1}^{k-1}(-1)^{i+j+1}\left|\begin{array}{c}
Y_{1..\hat{j}....\hat{k}....n}^{2,i-1}\\
X_{1...\hat{j}...\hat{k}...n}^{i+1,n}\end{array}\right|(y_{ij}-x_{ij})\right.\\
 &  & \left.+\sum_{i=2}^{n}\sum_{j=k+1}^{n}(-1)^{i+j}\left|\begin{array}{c}
Y_{1..\hat{k}....\hat{j}....n}^{2,i-1}\\
X_{1...\hat{k}...\hat{j}...n}^{i+1,n}\end{array}\right|(y_{ij}-x_{ij})\right).\end{eqnarray*}
 Here equality (1) comes from adding the extra term $\sum_{j=1}^{n}(-1)^{j+1}x_{1j}|X_{1...\hat{j}...n}^{2,n}|$
and then subtracting it from $\sum_{j=1}^{n}(-1)^{j+1}y_{1j}|X_{1...\hat{j}...n}^{2,n}|$$ $.
Notice that we have the third and forth sum as the following:\begin{eqnarray*}
 & {\displaystyle \sum_{k=1}^{n}}(-1)^{k+1}y_{1k}\left(\sum_{i=2}^{n}\sum_{j=1}^{k-1}(-1)^{i+j+1}\left|\begin{array}{c}
Y_{1..\hat{j}....\hat{k}....n}^{2,i-1}\\
X_{1...\hat{j}...\hat{k}...n}^{i+1,n}\end{array}\right|(y_{ij}-x_{ij})\right.\end{eqnarray*}
\[
\left.+{\displaystyle {\displaystyle \sum_{i=2}^{n}\sum_{j=k+1}^{n}}}(-1)^{i+j}\left|\begin{array}{c}
Y_{1..\hat{k}....\hat{j}....n}^{2,i-1}\\
X_{1...\hat{k}...\hat{j}...n}^{i+1,n}\end{array}\right|(y_{ij}-x_{ij})\right)\]
\[
={\displaystyle \sum_{i=2}^{n}{\displaystyle \sum_{j=1}^{n}}}{\displaystyle (-1)^{i+j}(y_{i,j}-x_{i,j})}\left({\displaystyle \sum_{k=1}^{j-1}(-1)^{k+1}y_{1k}}\left|\begin{array}{c}
Y_{1...\hat{k}..\hat{j}...n}^{2,i-1}\\
X_{1...\hat{k}..\hat{j}...n}^{i+1,n}\end{array}\right|\right.\]
\begin{eqnarray*}
 &  & \left.+{\displaystyle \sum_{k=j+1}^{n}(-1)^{k}y_{1k}\left|\begin{array}{c}
Y_{1...\hat{j}..\hat{k}...n}^{2,i-1}\\
X_{1...\hat{j}..\hat{k}...n}^{i+1,n}\end{array}\right|}\right)\\
 & = & \sum_{i=2}^{n}\sum_{j=1}^{n}(-1)^{i+j}\left|\begin{array}{c}
Y_{1...\hat{j}...n}^{1,i-1}\\
X_{1...\hat{j}...n}^{i+1,n}\end{array}\right|(y_{ij}-x_{ij}).\end{eqnarray*}
Now we can put everything together, we obtain \begin{eqnarray*}
|Y| & = & |X|+\sum_{j=1}^{n}(-1)^{j+1}|X_{1..\hat{j}...n}^{2,n}|(y_{1j}-x_{1j})\\
 &  & +\sum_{i=2}^{n}\sum_{j=1}^{n}(-1)^{i+j}\left|\begin{array}{c}
Y_{1...\hat{j}...n}^{1,i-1}\\
X_{1...\hat{j}...n}^{i+1,n}\end{array}\right|(y_{ij}-x_{ij})\\
 & = & |X|+{\displaystyle \sum_{i=1}^{n}\sum_{j=1}^{n}(-1)^{i+j}\left|\begin{array}{c}
Y_{1...\hat{j}...n}^{1,i-1}\\
X_{1...\hat{j}...n}^{j+1,n}\end{array}\right|(y_{ij}-x_{ij})}.\end{eqnarray*}
 
\end{proof}
In the following lemma, we define those special equations that we
consider and we show that those equations are in the defining ideal
of symmetric algebra of $\mathbb{D}$.
\begin{lem}
Let $X_{a_{1}...a_{s_{1}}}$ be the $s_{1}$ by $s_{1}$ submatrix
of $X_{s_{1}t_{1}}$ with columns $a_{1},...,a_{s_{1}}$, $Y_{b_{1}...b_{s_{2}}}$
the $s_{2}$ by $s_{2}$ submatrix of $Y_{s_{2}t_{2}}$ with columns
$b_{1},...,b_{s_{2}}$, $X_{a_{1}...a_{s_{1}}}^{l,k}$ the $k-l+1$
by $s_{1}$ submatrix of $X$ with rows $l,\, l+1,..,k$ and columns
$a_{1},...,a_{s_{1}}$, and similarly for $Y$ and $Z$. 

We define \[
g_{ij,lk}=\left|\begin{array}{cc}
z_{ij} & z_{lk}\\
x_{ij}-y_{ij} & x_{lk}-y_{lk}\end{array}\right|\]
\[
f_{a_{1},...,a_{s_{1}}}=\sum_{q=1}^{s_{2}}(-1)^{q+1}\left|\left[\begin{array}{c}
Z^{q,q}\\
Y^{1,q-1}\\
X^{q+1,m}\end{array}\right]_{a_{1}...a_{s_{1}}}\right|,\]
where $1\leq a_{1}<a_{2}<...<a_{s_{1}}\leq\mathrm{min}(t_{1},t_{2})$
and $1\leq i\leq m$, $1\leq l\leq m$, $1\leq j\leq n$, $1\leq k\leq n$. 

We write $\mathcal{L}=(I_{s_{1}}(X_{s_{1}t_{1}}),\, I_{s_{2}}(Y_{s_{2}t_{2}}),\, g_{_{ij,lk}},\, f_{a_{1},...,a_{s_{1}}})$,
which is an ideal of \\
$k[X_{mn},Y_{mn},Z_{mn}]$. Then $\mathcal{L}\subset\mathcal{J}$.\label{defineL}\end{lem}
\begin{proof}
We can see $|X_{a_{1}...a_{s_{1}}}|$, $|Y_{b_{1}...b_{s_{2}}}|$,
$g_{ij,lk}$'s are in $\mathcal{J}$. Notice that when $t_{2}<s_{1}$,
by the way we define $f_{a_{1},...,a_{s_{1}}}$, this is an empty
condition. When $t_{2}\geq s_{1}$, $ $ we substitute $z_{ij}$ via
$x_{ij}-y_{ij}$ and use Lemma \ref{YtoX}, we can see $f$'s are
in $\mathcal{J}$. 
\end{proof}
The next Theorem immediate implies Theorem \ref{linear}.
\begin{thm}
\label{LIsRD}The ideal $\mathcal{L}$ is the defining ideal of $\mathcal{R}(\mathbb{D})$
and $\mathbb{D}$ is of linear type.
\end{thm}
The following lemma is the key step to prove the main theorem. The
proof of the lemma is given in the Section 4. It involves finding
a Gröbner basis of the ideal.
\begin{lem}
\label{NZD}The variable $x_{11}$ is a non zero-divisor of the quotient
ring $k[X,Y,Z]/\mathcal{L}$.
\end{lem}
\medskip{}

\begin{flushleft}
$Proof\: of\: Theorem$ \ref{LIsRD}: From Lemma \ref{defineL}, we
have $\mathcal{L}\subset\mathcal{J}\subset\mathcal{K}$, where $\mathcal{J}$
is the defining ideal of $\mbox{Sym}(\mathbb{D})$. We would like
to show $\mathcal{L}=\mathcal{K}$ and as a consequence, $\mathcal{L}=\mathcal{J}=\mathcal{K}$,
i.e. $\mathbb{D}$ is an ideal of linear type. By Lemma \ref{NZD},
$x_{11}$ is a non zero-divisor on $k[X,Y,Z]/\mathcal{L}$. Changing
the roles of $X$ and $Y$, we also obtain that $y_{11}$ is a non-zero
divisor on $k[X,Y,Z]/\mathcal{L}$. Since $\mathcal{K}$ is a prime
ideal, it suffices to show that $\mathcal{L}_{x_{11}y_{11}}=\mathcal{K}_{x_{11}y_{11}}$.
The latter holds by induction on the size of the matrix $X$. 
\par\end{flushleft}

\begin{flushleft}
To explain this, we consider the $(m-1)\times(n-1)$ matrices of variables
$X'=[x_{ij}']$ , $Y'=[y_{ij}]$, $Z'=[z_{ij}]$, $2\leq i\leq m,$
$2\leq j\leq n$. We define a natural isomorphism $\phi$ from $k[\{x_{ij}\}_{i=1\mbox{ or }j=1},X']_{x_{11}}$
to $k[X]_{x_{11}}$ via $\phi(x_{ij})=x_{ij}$ when $i=1$ or $j=1$,
and $\phi(x_{ij}')=x_{ij}-x_{i1}x_{1j}/x_{11}$ when $i\neq1$ and
$j\neq1$. Let \[
R_{1}'=k[\{x_{ij}\}_{i=1\mbox{ or }j=1},X']_{x_{11}}/I'\cong R_{1x_{11}},\]
 \[
R_{2}'=k[\{x_{ij}\}_{i=1\mbox{ or }j=1},X']_{x_{11}}/J'\cong R_{2x_{11}},\]
 where $I'=I_{s_{1}-1}(X'_{s_{1}-1,t_{1}-1})$, $J'=I_{s_{2}-1}(X'_{s_{2}-1,t_{2}-1})$
and \begin{eqnarray*}
 &  & S':=k[\{x_{ij}\}_{i=1\mbox{ or }j=1},X',\{y_{ij}\}_{i=1\mbox{ or }j=1},Y']_{x_{11}y_{11}}/(I',J')\\
 &  & \cong R_{1}'\otimes R_{2}'\cong R_{1x_{11}}\otimes R_{2x_{11}}.\end{eqnarray*}
Then we have $ $\begin{eqnarray*}
\widetilde{\mathbb{D}}' & = & (\{x_{ij}-y_{ij}\}_{i=1\mbox{ or }j=1},\{x_{ij}'-y_{ij}'\}_{2\leq i\leq m,\mbox{ }2\leq j\leq n})\\
 & \cong & \mathbb{D}=(\{x_{ij}-y_{ij}\}_{1\leq i\leq m,\mbox{ }1\leq j\leq n}),\end{eqnarray*}
 and $T'=S'[\{z_{ij}\}_{i=1\mbox{ or }j=1},Z']\cong T_{x_{11}y_{11}}$
by the map $\overline{\phi}$ defined by $\overline{\phi}(x_{ij})=x_{ij}$,
$\overline{\phi}(y_{ij})=y_{ij}$ and $\overline{\phi}(z_{ij})=z_{ij}$
when $i=1$ or $j=1$, and $\overline{\phi}(x_{ij}')=x_{ij}-x_{i1}x_{1j}/x_{11}$,
$\overline{\phi}(y_{ij}')=y_{ij}-y_{i1}y_{1j}/y_{11}$, and $\overline{\phi}(z_{ij}')=z_{ij}-y_{i1}z_{1j}/y_{11}-x_{1j}z_{i1}/y_{11}+x_{i1}x_{1j}z_{11}/x_{11}y_{11}$
when $i\neq1$ and $j\neq1$.$ $ Let $\phi'$ denote the induced
map of $\overline{\phi}$ from $\mathcal{R}_{S'}(\widetilde{\mathbb{D}}')$
to $\mathcal{R}_{S_{x_{11}y_{11}}}(\mathbb{D})$. Let $\psi$ and
$\psi'$ denote the map from $T_{x_{11}y_{11}}$to $\mathcal{R}_{S_{x_{11}y_{11}}}(\mathbb{D})$
and $T'$ to $\mathcal{R}_{S'}(\widetilde{\mathbb{D}}')$. We obtain
the following diagram\[
\begin{array}{ccccc}
 &  & \overline{\phi}\\
 & T' & \longrightarrow & T_{x_{11}y_{11}}\\
 & \psi'\downarrow &  & \downarrow\psi\\
 & \mathcal{R}_{S'}(\mathbb{D}') & \longrightarrow & \mathcal{R}_{S_{x_{11}y_{11}}}(\mathbb{D})\\
 &  & \phi'\end{array}\]
\[
\]
To show the diagram commute, it is enough to show $ $$\phi'(\psi'(z_{ij}'))=\psi(\overline{\phi}(z_{ij}'))$.
Since \begin{eqnarray*}
\psi(\overline{\phi}(z_{ij}')) & = & \psi(z_{ij}-y_{i1}z_{1j}/y_{11}-x_{1j}z_{i1}/y_{11}+x_{i1}x_{1j}z_{11}/x_{11}y_{11})\\
 & = & x_{ij}-y_{ij}-y_{i1}(x_{1j}-y_{1j})/y_{11}-x_{1j}(x_{i1}-y_{i1})/y_{11}\\
 &  & +x_{i1}x_{1j}(x_{11}-y_{11})/x_{11}y_{11},\end{eqnarray*}
 and \begin{eqnarray*}
\phi'(\psi'(z_{ij}')) & = & \phi'(x_{ij}'-y_{ij}')\\
 & = & x_{ij}-y_{ij}-x_{i1}x_{1j}/x_{11}+y_{i1}y_{1j}/y_{11}\\
 & = & x_{ij}-y_{ij}-y_{i1}(x_{1j}-y_{1j})/y_{11}-x_{1j}(x_{i1}-y_{i1})/y_{11}\\
 &  & +x_{i1}x_{1j}(x_{11}-y_{11})/x_{11}y_{11}.\end{eqnarray*}
Hence $\phi'$ is an isomorphism. Let $\mathbb{D}'=(\{x'_{ij}-y'_{ij}\}_{2\leq i\leq m,\mbox{ }2\leq j\leq n})$
then by the induction hypothesis, the defining ideal of $\mathcal{R}_{S'}(\mathbb{D}')$
in $T'$ is of the form $\mathcal{L}'=\{I_{s_{1}-1}(X'_{s_{1}-1,t_{1}-1}),\, I_{s_{2}-1}(Y'_{s_{2}-1,t_{2}-1}),g'_{ij,lk},f'_{a_{2},...,a_{s-1}}\}$,
where \[
g'_{ij,lk}=\left|\begin{array}{cc}
z'_{ij} & z'_{lk}\\
x'_{ij}-y'_{ij} & x'_{lk}-y'_{lk}\end{array}\right|\]
 \[
f'_{a_{2},...,a_{s_{1}}}=\sum_{q=2}^{s_{2}}(-1)^{q+1}\left|\left[\begin{array}{c}
Z'^{q,q}\\
Y'^{2,q-1}\\
X'^{q+1,m}\end{array}\right]_{a_{2}...a_{s_{1}}}\right|\]
 with $2\leq a_{1}<a_{2}<...<a_{s_{1}}\leq\mathrm{min}\{t_{1},t_{2}\}$
and $2\leq i\leq m$, $2\leq l\leq m$, $2\leq j\leq n$, $2\leq k\leq n$.
Let $W$ denote the set of Koszul relations: \[
g^{1}{}_{ij,lk}=\left|\begin{array}{cc}
z{}_{ij} & z'_{lk}\\
x{}_{ij}-y{}_{ij} & x'_{lk}-y'_{lk}\end{array}\right|\]
 with $i=1$ or $j=1$ and \[
g^{2}{}_{ij,lk}=\left|\begin{array}{cc}
z{}_{ij} & z{}_{lk}\\
x{}_{ij}-y{}_{ij} & x{}_{lk}-y{}_{lk}\end{array}\right|\]
with $i=1$ or $j=1$ and $l=1$ or $k=1$. Then $(\mathcal{L}',W)$
is the defining ideal of $\mathcal{R}_{S'}(\widetilde{\mathbb{D}}')\cong\mathcal{R}_{S_{x_{11}y_{11}}}(\mathbb{D})$.
Once we show that $\overline{\phi}(\mathcal{L}',W)\subset\mathcal{L}_{x_{11}y_{11}}$,
then $\mathcal{L}_{x_{11}y_{11}}=\mathcal{K}_{x_{11}y_{11}}$.
\par\end{flushleft}

From the way we define the map $\overline{\phi}$, we have \[
\overline{\phi}(I_{s_{1}-1}(X'_{s_{1}-1,t_{1}-1}),\, I_{s_{2}-1}(Y'_{s_{2}-1,t_{2}-1}),g'_{ij,lk},W)\subset\mathcal{L}_{x_{11}y_{11}}.\]
 Notice the following equality:\begin{eqnarray*}
\overline{\phi}(f'_{a_{2},...,a_{s_{1}}}) & = & \frac{1}{y_{11}}f{}_{1,a_{2},...,a_{s_{1}}}-\frac{z_{11}}{x_{11}y_{11}}|X_{1a_{2}...a_{s_{1}}}|,\end{eqnarray*}
hence $\overline{\phi}(f'_{a_{2},...,a_{s_{1}}})\in\mathcal{L}_{x_{11}y_{11}}$.

\begin{flushleft}
\hfill{}$\square$
\par\end{flushleft}

\section{SOME LINEAR ALGEBRA}

This section is a reminder of some linear algebra properties. We will
use those properties in the proof of Section 4.

The following lemma writes the determinant of a certain in $x$ and
$y$ variables in term of $y$ variables and differences $x_{ij}-y_{ij}$.
\begin{lem}
\label{xtox-y}With notation as \ref{SubMatrix}, \[
\left|\begin{array}{cccccc}
 &  & Y_{1,...,n}^{1,i-1}\\
y_{i,1} & ... & y_{i,j} & x_{i,j+1} & ... & x_{i,n}\\
 &  & X_{1,...,n}^{i+1,n}\end{array}\right|\]
\textup{\begin{eqnarray*}
 & = & |Y|+{\displaystyle {\displaystyle \sum_{k=j+1}^{n}}(-1)^{i+k}\left|\begin{array}{c}
Y_{1..\hat{k}..n}^{1,i-1}\\
X_{1...\hat{k}...n}^{i+1,n}\end{array}\right|(x_{ik}-y_{ik})}\\
 &  & +{\displaystyle \sum_{l=i+1}^{n}{\displaystyle \sum_{k=1}^{n}(-1)^{l+k}\left|\begin{array}{c}
Y_{1...\hat{k}...n}^{1,l-1}\\
X_{1...\hat{k}...n}^{l+1,n}\end{array}\right|(x_{lk}-y_{lk}})}.\end{eqnarray*}
}\end{lem}
\begin{proof}
We will show this by inducting on $i$. When $i=1$, we have\begin{eqnarray*}
 & \left|\begin{array}{c}
y_{1,1}...y_{1,j}x_{1,j+1}..x_{1,n}\\
X^{2,n}\end{array}\right|\\
= & {\displaystyle \sum_{k=1}^{j}}(-1)^{k+1}y_{1,k}|X_{1...\hat{k}...n}^{2,n}|+{\displaystyle \sum_{k=j+1}^{n}}(-1)^{k+1}x_{1,k}|X_{1...\hat{k}...n}^{2,n}|\end{eqnarray*}
\begin{eqnarray*}
= & {\displaystyle \sum_{k=1}^{j}}(-1)^{k+1}(y_{1,k}-x_{1,k})|X_{1...\hat{k}...n}^{2,n}|+{\displaystyle \sum_{k=1}^{n}}(-1)^{k+1}x_{1,k}|X_{1...\hat{k}...n}^{2,n}|\\
= & {\displaystyle \sum_{k=1}^{j}}(-1)^{k+1}(y_{1,k}-x_{1,k})|X_{1...\hat{k}...n}^{2,n}|+|X| & .\end{eqnarray*}
Now Lemma \ref{YtoX} gives \[
|X|=|Y|+{\displaystyle \sum_{l=1}^{n}\sum_{j=1}^{n}(-1)^{l+j}\left|\begin{array}{c}
Y_{1...\hat{j}...n}^{1,l-1}\\
X_{1...\hat{j}...n}^{l+1,n}\end{array}\right|(x_{lj}-y_{lj})}.\]
 Hence we have \begin{eqnarray*}
 & \left|\begin{array}{c}
y_{1,1}...y_{1,j}x_{1,j+1}..x_{1,n}\\
X^{2,n}\end{array}\right|\\
= & {\displaystyle \sum_{k=1}^{j}}(-1)^{k+1}(y_{1k}-x_{1k})|X_{1...\hat{k}...n}^{2,n}|+|Y|+{\displaystyle \sum_{l=1}^{n}\sum_{j=1}^{n}(-1)^{l+j}\left|\begin{array}{c}
Y_{1...\hat{j}...n}^{1,l-1}\\
X_{1...\hat{j}...n}^{l+1,n}\end{array}\right|(x_{lj}-y_{lj})}.\end{eqnarray*}
 Notice that\[
{\displaystyle \sum_{k=1}^{j}(-1)^{k+1}(y_{1k}-x_{1k})|X_{1...\hat{k}...n}^{2,n}|+\sum_{l=1}^{n}\sum_{j=1}^{n}(-1)^{l+j}\left|\begin{array}{c}
Y_{1...\hat{j}...n}^{1,l-1}\\
X_{1...\hat{j}...n}^{l+1,n}\end{array}\right|(x_{lj}-y_{lj})}\]
\begin{eqnarray*}
 & = & \sum_{k=1}^{j}(-1)^{k+1}(y_{1k}-x_{1k})|X_{1...\hat{k}...n}^{2,n}|+\sum_{j=1}^{n}(-1)^{j+1}|X_{1...\hat{j}...n}^{2,n}|(x_{1j}-y_{1j})\\
 & {\displaystyle } & +{\displaystyle \sum_{l=2}^{n}\sum_{j=1}^{n}(-1)^{l+j}\left|\begin{array}{c}
Y_{1...\hat{j}...n}^{1,l-1}\\
X_{1...\hat{j}...n}^{l+1,n}\end{array}\right|(x_{lj}-y_{lj})}\end{eqnarray*}
\begin{eqnarray*}
= & {\displaystyle \sum_{k=j+1}^{n}}(-1)^{k+1}(x_{1k}-y_{1k})|X_{1...\hat{k}...n}^{2,n}|+{\displaystyle \sum_{l=2}^{n}\sum_{j=1}^{n}}(-1)^{l+j}\left|\begin{array}{c}
Y_{1...\hat{j}...n}^{1,l-1}\\
X_{1...\hat{j}...n}^{l+1,n}\end{array}\right|(x_{lj}-y_{lj}) & .\end{eqnarray*}
Hence \begin{eqnarray*}
\left|\begin{array}{c}
y_{1,1}...y_{1,j}x_{1,j+1}..x_{1,n}\\
X^{2,n}\end{array}\right| & = & |Y|+\sum_{k=j+1}^{n}(-1)^{k+1}(x_{1,k}-y_{1,k})|X_{1...\hat{k}...n}^{2,n}|\\
 &  & +\sum_{l=2}^{n}\sum_{j=1}^{n}(-1)^{l+j}\left|\begin{array}{c}
Y_{1...\hat{j}...n}^{1,l-1}\\
X_{1...\hat{j}...n}^{l+1,n}\end{array}\right|(x_{lj}-y_{lj}).\end{eqnarray*}
Now for the induction step, we assume \[
\left|\begin{array}{c}
Y^{1,i-1}\\
X^{i,n}\end{array}\right|=|Y|+\sum_{l=i}^{n}\sum_{j=1}^{n}(-1)^{l+j}\left|\begin{array}{c}
Y_{1...\hat{j}...n}^{1,l-1}\\
X_{1...\hat{j}...n}^{l+1,n}\end{array}\right|(x_{lj}-y_{lj}).\]
Therefore \[
\begin{array}{ccc}
 & \left|\begin{array}{c}
Y^{1,i-1}\\
y_{i,1}...y_{i,j}x_{i,j+1}..x_{i,n}\\
X^{i+1,n}\end{array}\right|\\
= & {\displaystyle \sum_{k=1}^{j}}(-1)^{i+k}y_{i,k}\left|\begin{array}{c}
Y_{1...\hat{k}...n}^{1,i-1}\\
X_{1...\hat{k}...n}^{i+1,n}\end{array}\right|+{\displaystyle \sum_{k=j+1}^{n}}(-1)^{i+k}x_{i,k}\left|\begin{array}{c}
Y_{1...\hat{k}...n}^{1,i-1}\\
X_{1...\hat{k}...n}^{i+1,n}\end{array}\right|\\
= & {\displaystyle \sum_{k=1}^{j}}(-1)^{i+k}(y_{i,k}-x_{i,k})\left|\begin{array}{c}
Y_{1...\hat{k}...n}^{1,i-1}\\
X_{1...\hat{k}...n}^{i+1,n}\end{array}\right|+{\displaystyle \sum_{k=1}^{n}}(-1)^{i+k}x_{i,k}\left|\begin{array}{c}
Y_{1...\hat{k}...n}^{1,i-1}\\
X_{1...\hat{k}...n}^{i+1,n}\end{array}\right|\end{array}\]
\[
={\displaystyle \sum_{k=1}^{j}}(-1)^{i+k}(y_{i,k}-x_{i,k})\left|\begin{array}{c}
Y_{1...\hat{k}...n}^{1,i-1}\\
X_{1...\hat{k}...n}^{i+1,n}\end{array}\right|+\left|\begin{array}{c}
Y^{1,i-1}\\
X^{i,n}\end{array}\right|\]
\begin{eqnarray*}
= & {\displaystyle \sum_{k=1}^{j}}(-1)^{i+k}(y_{i,k}-x_{i,k})\left|\begin{array}{c}
Y_{1...\hat{k}...n}^{1,i-1}\\
X_{1...\hat{k}...n}^{i+1,n}\end{array}\right|+|Y|\\
 & +{\displaystyle \sum_{l=i}^{n}\sum_{j=1}^{n}}(-1)^{l+j}\left|\begin{array}{c}
Y_{1...\hat{j}...n}^{1,l-1}\\
X_{1...\hat{j}...n}^{l+1,n}\end{array}\right|(x_{lj}-y_{lj})\end{eqnarray*}
\begin{eqnarray*}
 & = & |Y|+\sum_{k=j+1}^{n}(-1)^{i+k}(x_{i,k}-y_{i,k})\left|\begin{array}{c}
Y_{1...\hat{k}...n}^{1,i-1}\\
X_{1...\hat{k}...n}^{i+1,n}\end{array}\right|\\
 &  & +\sum_{l=i+1}^{n}\sum_{j=1}^{n}(-1)^{l+j}\left|\begin{array}{c}
Y_{1...\hat{j}...n}^{1,l-1}\\
X_{1...\hat{j}...n}^{l+1,n}\end{array}\right|(x_{lj}-y_{lj}).\end{eqnarray*}

\end{proof}
By the following the two lemmas, the $S$-pairs between the elements
of the ideal $\mathcal{L}$ can be reduced using the Koszul relations. 
\begin{lem}
\label{g_ij} Let $1\leq i,\: l\leq m$, $1\leq j,\: k\leq n$, $a_{1}<a_{2}<a_{3}$.
Let \[
g_{ij,lk}=\left|\begin{array}{cc}
z_{ij} & z_{lk}\\
x_{ij}-y_{ij} & x_{lk}-y_{lk}\end{array}\right|,\, M=\left|\begin{array}{ccc}
z_{1a_{1}} & z_{1a_{2}} & z_{1a_{3}}\\
x_{1a_{1}} & x_{1a_{2}} & x_{1a_{3}}\\
y_{1a_{1}} & y_{1a_{2}} & y_{1a_{3}}\end{array}\right|.\]
 Then \[
M=y_{1a_{1}}g_{1a_{2},1a_{3}}-y_{1a_{2}}g_{1a_{1},1a_{3}}+y_{1a_{3}}g_{1a_{1},1a_{2}}.\]
\end{lem}
\begin{proof}
One has \begin{eqnarray*}
M & = & \left|\begin{array}{ccc}
z_{1a_{1}} & z_{1a_{2}} & z_{1a_{3}}\\
x_{1a_{1}}-y_{1a_{1}} & x_{1a_{2}}-y_{1a_{2}} & x_{1a_{3}}-y_{1a_{3}}\\
y_{1a_{1}} & y_{1a_{2}} & y_{1a_{3}}\end{array}\right|\end{eqnarray*}
\begin{eqnarray*}
 & = & y_{1a_{1}}\left|\begin{array}{cc}
z_{1a_{2}} & z_{1a_{3}}\\
x_{1a_{2}}-y_{1a_{2}} & x_{1a_{3}}-y_{1a_{3}}\end{array}\right|-y_{1a_{2}}\left|\begin{array}{cc}
z_{1a_{1}} & z_{1a_{3}}\\
x_{1a_{1}}-y_{a_{1}} & x_{1a_{3}}-y_{1a_{3}}\end{array}\right|\\
 &  & +y_{1a_{3}}\left|\begin{array}{cc}
z_{1a_{1}} & z_{1a_{2}}\\
x_{1a_{1}}-y_{1a_{1}} & x_{1a_{2}}-y_{1a_{2}}\end{array}\right|\end{eqnarray*}
\end{proof}
\begin{lem}
\label{Swithchg_ij} Let $g_{ij,lk}$ defined as Lemma \ref{g_ij}.
Then \begin{eqnarray*}
M & = & \left|\begin{array}{cc}
z_{1a_{1}} & z_{1a_{2}}\\
x_{2a_{1}}-y_{2a_{1}} & x_{2a_{2}}-y_{2a_{2}}\end{array}\right|\\
 & = & g_{1a_{1},2a_{2}}-g_{1a_{2},2a_{1}}+\left|\begin{array}{cc}
x_{1a_{1}}-y_{1a_{1}} & x_{1a_{2}}-y_{1a_{2}}\\
z_{2a_{1}} & z_{2a_{2}}\end{array}\right|.\end{eqnarray*}
\end{lem}
\begin{proof}
We have \begin{eqnarray*}
M & = & z_{1a_{1}}(x_{2a_{2}}-y_{2a_{2}})-z_{1a_{2}}(x_{2a_{1}}-y_{2a_{1}})\\
 & = & g_{1a_{1},2a_{2}}+z_{2a_{2}}(x_{1a_{1}}-y_{1a_{1}})-g_{1a_{2},2a_{1}}-z_{2a_{1}}(x_{1a_{2}}-y_{1a_{2}})\\
 & = & g_{1a_{1},2a_{2}}-g_{1a_{2},2a_{1}}+\left|\left[\begin{array}{cc}
x_{1a_{1}}-y_{1a_{1}} & x_{1a_{2}}-y_{1a_{2}}\\
z_{2a_{1}} & z_{2a_{2}}\end{array}\right]\right|.\end{eqnarray*}

\end{proof}

\section{Gröbner basis}

This section is devoted to prove Lemma \ref{NZD}. We outline this
section here. We will recall Buchberger's Criterion and give several
lemmas that will help us reduce the computation of $S$-pairs between
elements of $\mathcal{L}$. We define several equations and show those
equations sit inside the ideal $\mathcal{L}$. Theorem \ref{GB} will
give a Gröbner basis of $\mathcal{L}$ via a particular ordering.
Actually, all the equations defined before Theorem \ref{GB} are all
of elements of the Gröbner basis. The proof of Theorem \ref{GB} will
be broken down as several lemmas computing the $S$-pairs of the elements
and showing all of the reminders of $S$-pairs are zero. 

Let $I=(g_{1},...,g_{s})$ be an ideal in a polynomial ring. We define
\[
\mbox{in}(g_{j})/\mathrm{GCD}(\mbox{in}(g_{i}),\mbox{in}(g_{j}))=m_{ji},\]
 \[
\mbox{in}(g_{i})/\mathrm{GCD}(\mbox{in}(g_{i}),\mbox{in}(g_{j}))=m_{ij},\]
 and \[
m_{ji}g_{i}-m_{ij}g_{j}=\sum f_{u}^{(ij)}g_{u}+h_{g_{i}g_{j}}\]
 where $\mbox{in}(m_{ji}g_{i})>\mbox{in}(f_{u}^{(ij)}g_{u})$ for
all $u$.
\begin{thm}
$\mathrm{(Buchberger's\: Criterion)}$. The elements $g_{1},...,g_{s}$
form a Gröbner basis if and only if $h_{g_{i}g_{j}}=0$ for all $i$
and $j$.
\end{thm}
The polynomial $m_{ji}g_{i}-m_{ij}g_{j}$ is commonly referred to
as the $S$-pair between $g_{i}$ and $g_{j}$ and $h_{g_{i}g_{j}}$
is called the remainder. 

By using Buchberger's Criterion, we obtain several lemmas that will
help in the computation of a Gröbner basis of $\mathcal{L}$. Sine
we focus on the determinantal rings, the computation of $S$-pair
between elements are involving the values of matrix determinate. For
the computation purpose, we define the following definition. 
\begin{defn}
\label{M_ij}Given two square free monomials $p_{1}$ and $p_{2}$
in $k[X]$ where $X$ is the $m$ by $n$ matrix of variables, we
define $m_{12}=p_{1}/\mbox{GCD}(p_{1},p_{2})$ and $m_{21}=p_{2}/\mbox{GCD}(p_{1},p_{2})$.
Assume $ $$m_{12}=x_{u_{1}a_{1}}...x_{u_{r}a_{r}}$, $m_{21}=x_{v_{1}b_{1}}...x_{v_{w}b_{w}}$,
then define the matrix \begin{eqnarray*}
M_{12}:=\left[\begin{array}{ccc}
x_{u_{1}a_{1}} & ... & x_{u_{1}a_{r}}\\
x_{u_{2}a_{1}} & ... & x_{u_{2}a_{r}}\\
\vdots &  & \vdots\\
x_{u_{r}a_{1}} &  & x_{u_{r}a_{r}}\end{array}\right]\end{eqnarray*}
and the matrix \begin{eqnarray*}
M_{21}:=\left[\begin{array}{ccc}
x_{v_{1}b_{1}} & ... & x_{v_{1}b_{w}}\\
x_{v_{2}b_{1}} & ... & x_{v_{2}b_{w}}\\
\vdots &  & \vdots\\
x_{v_{w}b_{1}} &  & x_{v_{w}b_{w}}\end{array}\right] &  & .\end{eqnarray*}

\end{defn}
\medskip{}

The following lemma helps us replace a polynomial with a leading term
involving $x_{i,j}$'s by a polynomial with a leading term without
involving $x_{ij}$'s.
\begin{lem}
\label{xTox_y}Let $a_{s_{1}}<...<a_{1}$, $1\leq r\leq s_{1}$, and
let $g_{i_{1}j_{1},i_{2}j_{2}}$ be as defined in Lemma \ref{defineL}.
Then

\[
\left|\left[\begin{array}{c}
Y^{1,r-1}\\
Z^{r,r}\\
X^{r+1,s_{1}}\end{array}\right]_{a_{s_{1}},...,a_{1}}\right|=\left|\left[\begin{array}{c}
Y^{1,r-1}\\
Z^{r,r}\\
Y^{r+1,s_{1}}\end{array}\right]_{a_{s_{1}},...,a_{1}}\right|+\sum_{u=r+1}^{s_{1}}\left|\left[\begin{array}{c}
Y^{1,r-1}\\
X^{r,r}-Y^{r,r}\\
Y^{r+1,u-1}\\
Z^{u,u}\\
X^{u+1,s_{1}}\end{array}\right]_{a_{s_{1}},...,a_{1}}\right|\]
\begin{eqnarray*}
\end{eqnarray*}
\begin{eqnarray*}
 & +{\displaystyle \sum_{u=r+1}^{s_{1}}}{\displaystyle \sum_{\{c_{1},c_{2},d_{1},...,d_{s_{1}-2}\}=\{a_{1},...,a_{s_{1}}\}}}\pm(g_{rc_{1},uc_{2}}-g_{rc_{2},uc_{1}})\left|\left[\begin{array}{c}
Y^{1,r-1}\\
Y^{r+1,u-1}\\
X^{u+1,s_{1}}\end{array}\right]_{d_{1},...,d_{s_{1}-2}}\right|.\end{eqnarray*}
\end{lem}
\begin{proof}
For the purpose of this proof we drop the column indices. We use Lemma
\ref{xtox-y} to obtain\[
\left|\left[\begin{array}{c}
Y^{1,r-1}\\
Z^{r,r}\\
X^{r+1,s_{1}}\end{array}\right]\right|=\left|\left[\begin{array}{c}
Y^{1,r-1}\\
Z^{r,r}\\
Y^{r+1,s_{1}}\end{array}\right]\right|+\sum_{u=r+1}^{s_{1}}\left|\left[\begin{array}{c}
Y^{1,r-1}\\
Z^{r,r}\\
Y^{r+1,u-1}\\
X^{u,u}-Y^{u,u}\\
X^{u+1,s_{1}}\end{array}\right]\right|.\]
Notice that

\begin{eqnarray*}
 & \left|\left[\begin{array}{c}
Y^{1,r-1}\\
Z^{r,r}\\
Y^{r+1,u-1}\\
X^{u,u}-Y^{u,u}\\
X^{u+1,s_{1}}\end{array}\right]\right|\\
= & {\displaystyle \sum_{\{c_{1},c_{2},d_{1},...,d_{s_{1}-2}\}=\{a_{1},...,a_{s_{1}}\}}}\\
 & \pm\left|\left[\begin{array}{cc}
z_{rc_{1}} & z_{rc_{2}}\\
x_{uc_{1}}-y_{uc_{1}} & x_{uc_{2}}-y_{uc_{2}}\end{array}\right]\right|\left|\left[\begin{array}{c}
Y^{1,r-1}\\
Y^{r+1,u-1}\\
X^{u+1,s_{1}}\end{array}\right]_{d_{1},...,d_{s_{1}-2}}\right|\end{eqnarray*}
\begin{eqnarray*}
= & {\displaystyle \sum_{\{c_{1},c_{2},d_{1},...,d_{s_{1}-2}\}=\{a_{1},...,a_{s_{1}}\}}}\pm\Biggl(g_{rc_{1},uc_{2}}-g_{rc_{2},uc_{1}}\\
 & \left.+\left|\left[\begin{array}{cc}
x_{rc_{1}}-y_{rc_{1}} & x_{rc_{2}}-y_{rc_{2}}\\
z_{uc_{1}} & z_{uc_{2}}\end{array}\right]\right|\right)\left|\left[\begin{array}{c}
Y^{1,r-1}\\
Y^{r+1,u-1}\\
X^{u+1,s_{1}}\end{array}\right]_{d_{1},...,d_{s_{1}-2}}\right|\\
= & {\displaystyle \sum_{u=r+1}^{s_{1}}\sum_{\{c_{1},c_{2},d_{1},...,d_{s_{1}-2}\}=\{a_{1},...,a_{s_{1}}\}}}\\
 & \pm(g_{rc_{1},uc_{2}}-g_{rc_{2},uc_{1}})\left|\left[\begin{array}{c}
Y^{1,r-1}\\
Y^{r+1,u-1}\\
X^{u+1,s_{1}}\end{array}\right]_{d_{1},...,d_{s_{1}-2}}\right|\\
 & +\sum_{u=r+1}^{s_{1}}\left|\left[\begin{array}{c}
Y^{1,r-1}\\
X^{r,r}-Y^{r,r}\\
Y^{r+1,u-1}\\
Z^{u,u}\\
X^{u+1,s_{1}}\end{array}\right]\right| & .\end{eqnarray*}

\medskip{}

\end{proof}
The determinant in the following lemma appears in many cases in the
computing of Gröbner basis. This lemma enables the determinant to
be written as a combination of elements of $I_{s_{2}}(Y)$ and $g_{i_{1}j_{1},i_{2}j_{2}}$.
\begin{lem}
\label{topx}Let $a_{1}<...<a_{s_{1}+1}$, and $1\leq r\leq s_{1}$.
One has \[
\sum_{u=r}^{s_{2}}\left|\left[\begin{array}{c}
X^{r,r}\\
Y^{1,u-1}\\
Z^{u,u}\\
X^{u+1,s_{1}}\end{array}\right]_{a_{1},...,a_{s_{1}+1}}\right|\in I_{s_{2}}(Y)+(g_{i_{1}j_{1},i_{2}j_{2}}|\,1\leq i_{v}\leq m,\,1\leq j_{v}\leq n,\mbox{}v=1,2).\]
\end{lem}
\begin{proof}
The column indices are omitted again. First we write \begin{eqnarray*}
 & {\displaystyle \sum_{u=r}^{s_{2}}\left|\left[\begin{array}{c}
X^{r,r}\\
Y^{1,u-1}\\
Z^{u,u}\\
X^{u+1,s_{1}}\end{array}\right]\right|=\left|\left[\begin{array}{c}
X^{r,r}\\
Y^{1,r-1}\\
Z^{r,r}\\
X^{r+1,s_{1}}\end{array}\right]\right|+{\displaystyle \sum_{u=r+1}^{s_{2}}\left|\left[\begin{array}{c}
X^{r,r}\\
Y^{1,u-1}\\
Z^{u,u}\\
X^{u+1,s_{1}}\end{array}\right]\right|}}\end{eqnarray*}
then using Lemma \ref{g_ij} and \ref{xTox_y}, we obtain \[
\left|\left[\begin{array}{c}
X^{r,r}\\
Y^{1,r-1}\\
Z^{r,r}\\
X^{r+1,s_{1}}\end{array}\right]\right|=\left|\left[\begin{array}{c}
X^{r,r}-Y^{r,r}\\
Y^{1,r-1}\\
Z^{r,r}\\
X^{r+1,s_{1}}\end{array}\right]\right|+\left|\left[\begin{array}{c}
Y^{r,r}\\
Y^{1,r-1}\\
Z^{r,r}\\
X^{r+1,s_{1}}\end{array}\right]\right|\]
\begin{eqnarray*}
= & {\displaystyle \sum_{\{c_{1},c_{2},d_{1},...,d_{s_{1}-1}\}=\{a_{1},...,a_{s_{1+1}}\}}}\pm g_{rc_{1},rc_{2}}\left|\left[\begin{array}{c}
Y^{1,r-1}\\
X^{r+1,s_{1}}\end{array}\right]_{c_{1},...,c_{s_{1}-1}}\right|+\left|\left[\begin{array}{c}
Y^{r,r}\\
Y^{1,r-1}\\
Z^{r,r}\\
Y^{r+1,s_{1}}\end{array}\right]\right|\end{eqnarray*}
\[
+{\displaystyle \sum_{u=r+1}^{s_{2}}\sum_{\{c_{1},c_{2},d_{1},...,d_{s_{1}-1}\}=\{a_{1},...,a_{s_{1}+1}\}}}\pm(g_{rc_{1},uc_{2}}-g_{rc_{2},uc_{1}})\left|\left[\begin{array}{c}
Y^{r,r}\\
Y^{1,r-1}\\
Y^{r+1,u-1}\\
X^{u+1,s_{1}}\end{array}\right]_{d_{1},...,d_{s_{1}-1}}\right|\]
\begin{eqnarray*}
 &  & +\sum_{u=r+1}^{s_{2}}\left|\left[\begin{array}{c}
Y^{r,r}\\
Y^{1,r-1}\\
X^{r,r}-Y^{r,r}\\
Y^{r+1,u-1}\\
Z^{u,u}\\
X^{u+1,s_{1}}\end{array}\right]\right|+\sum_{u=s_{2}+1}^{s_{1}}\left|\left[\begin{array}{c}
Y^{r,r}\\
Z^{r}\\
Y^{1,r-1}\\
Y^{r+1,u-1}\\
X^{u,u}-Y^{u,u}\\
X^{u+1,s_{1}}\end{array}\right]\right|.\end{eqnarray*}
 $\alpha$ is defined as given : \[
\alpha=\sum_{\{c_{1},c_{2},d_{1},...,d_{s_{1}-1}\}=\{a_{1},...,a_{s_{1+1}}\}}\pm g_{rc_{1},rc_{2}}\left|\left[\begin{array}{c}
Y^{1,r-1}\\
X^{r+1,s_{1}}\end{array}\right]_{c_{1},...,c_{s_{1}-1}}\right|+\left|\left[\begin{array}{c}
Y^{r,r}\\
Y^{1,r-1}\\
Z^{r,r}\\
Y^{r+1,s_{1}}\end{array}\right]\right|\]
\[
+\sum_{u=r+1}^{s_{2}}\sum_{\{c_{1},c_{2},d_{1},...,d_{s_{1}-1}\}=\{a_{1},...,a_{s_{1}+1}\}}\pm(g_{rc_{1},uc_{2}}-g_{rc_{2},uc_{1}})\left|\left[\begin{array}{c}
Y^{r,r}\\
Y^{1,r-1}\\
Y^{r+1,u-1}\\
X^{u+1,s_{1}}\end{array}\right]_{d_{1},...,d_{s_{1}-1}}\right|\]
\begin{eqnarray*}
 & +\sum_{u=s_{2}+1}^{s_{1}}\left|\left[\begin{array}{c}
Y^{r,r}\\
Z^{r}\\
Y^{1,r-1}\\
Y^{r+1,u-1}\\
X^{u,u}-Y^{u,u}\\
X^{u+1,s_{1}}\end{array}\right]\right|.\end{eqnarray*}
This shows the element $\alpha$ is in $I_{s_{2}}(Y_{s_{2}t_{2}})+(g_{i_{1}j_{1},i_{2}j_{2}})$.
After removing the repeated row $y_{r}$, we have\\
 \[
{\displaystyle \sum_{u=r+1}^{s_{2}}}\left|\left[\begin{array}{c}
Y^{r,r}\\
Y^{1,r-1}\\
X^{r,r}-Y^{r,r}\\
Y^{r+1,u-1}\\
Z^{u,u}\\
X^{u+1,s_{1}}\end{array}\right]\right|=\sum_{u=r+1}^{s_{2}}\left|\left[\begin{array}{c}
Y^{r,r}\\
Y^{1,r-1}\\
X^{r,r}\\
Y^{r+1,u-1}\\
Z^{u,u}\\
X^{u+1,s_{1}}\end{array}\right]\right|\]
\begin{eqnarray*}
= & -{\displaystyle \sum_{u=r+1}^{s_{2}}\left|\left[\begin{array}{c}
X^{r,r}\\
Y^{1,r-1}\\
Y^{r,r}\\
Y^{r+1,u-1}\\
Z^{u,u}\\
X^{u+1,s_{1}}\end{array}\right]\right|} & =-\sum_{u=r+1}^{s_{2}}\left|\left[\begin{array}{c}
X^{r,r}\\
Y^{1,u-1}\\
Z^{u,u}\\
X^{u+1,s_{1}}\end{array}\right]\right|.\end{eqnarray*}

\end{proof}
In order to simplify the notation and the computation, we define notation
to keep track of sums of determinants.
\begin{defn}
\label{TopM}Let $G$ be a collection of polynomials in the ring $k[X,Y,Z]$
with $X$, $Y$, $Z$ as $m$ by $n$ matrices of variables over the
field $k$. Let $\{P_{a_{1},...,a_{q_{u}}}^{u}\}_{u\in I}$ be an
element of $G$ such that each $P_{a_{1},...,a_{q_{u}}}^{u}$ is the
sum of determinants $P_{i}^{u}$ of $m$ by $n$ matrices with the
same column indices, $a_{1},...,a_{q_{u}}$, in variables $X$, $Y$
and $Z$. Denote $P_{a_{1},...,a_{q_{u}}}^{u}=\sum_{i=1}^{p_{u}}P_{i}^{u}$
with $P_{1}^{u}$ containing the leading term of $P_{a_{1},...,a_{q_{u}}}^{u}$.
For example the element $f_{a_{1},...,a_{s_{1}}}$ in the Lemma \ref{defineL}
is written as $f_{a_{1},...,a_{s_{1}}}=\sum_{i=1}^{s_{2}}f_{i}$.

Given $P_{a_{1},...,a_{q_{u}}}^{u}$ and $P_{b_{1},...,b_{q_{v}}}^{v}$
in $G$, define $m_{12}$, $m_{21}$, $M_{12}$ and $M_{21}$ as definition
\ref{M_ij}. Assume $M_{12}$ has column index $c_{1},...,c_{p_{12}}$
and $M_{21}$ has column indices, $d_{1},...,d_{p_{21}}$. Define
$\overline{P_{a_{1},..,a_{q_{u}},d_{1},...,d_{p_{21}}}^{u}}$ and
$\overline{P_{b_{1},...,b_{q_{v}},c_{1},...,c_{p_{12}}}^{v}}$ as
following: add the rows of $M_{21}$ on top of each matrix of $P_{a_{1},...,a_{q_{u}}}^{u}$
and add the columns of $M_{21}$ in front of each matrix of $P_{a_{1},...,a_{q_{u}}}^{u}$.
Take the determinant of each matrix and take the sum of all determinants
to form the new polynomial $\overline{P_{a_{1},..,a_{q_{u}},d_{1},...,d_{p_{21}}}^{u}}$.
Similarly, use $M_{12}$ to obtain $\overline{P_{b_{1},...,b_{q_{v}},c_{1},...,c_{p_{12}}}^{v}}$
. Write $\overline{P_{a_{1},..,a_{q_{u}},d_{1},...,d_{p_{21}}}^{u}}=\sum_{i=1}^{p_{u}}\overline{P_{i}^{u}}$
and $\overline{P_{b_{1},...,b_{q_{v}},c_{1},...,c_{p_{12}}}^{v}}=\sum_{i=1}^{p_{v}}\overline{P_{i}^{v}}$
where $\overline{P_{1}^{u}}$ and $\overline{P_{1}^{v}}$ contain
the leading terms of $\overline{P_{a_{1},..,a_{q_{u}},d_{1},...,d_{p_{21}}}^{u}}$
and $\overline{P_{b_{1},...,b_{q_{v}},c_{1},...,c_{p_{12}}}^{v}}$
. For example in Lemma \ref{defineL}, we have $f_{a_{1},...,a_{s_{1}}}=f_{a_{1},...,a_{s_{1}}}^{1}$
and $f_{b_{1},a_{2},...,a_{s_{1}}}=f_{b_{1}a_{2},...,a_{s_{1}}}^{2}$
then $m_{12}=M_{12}=z_{a_{1}}$ and $m_{21}=M_{21}=z_{b_{1}}$ then
\begin{eqnarray*}
\overline{f_{b_{1},a_{1},a_{2},...,a_{s_{1}}}^{1}} & = & \sum_{q=1}^{s_{2}}(-1)^{q+1}\left|\left[\begin{array}{c}
Z^{1,1}\\
Z^{q,q}\\
Y^{1,q-1}\\
X^{q+1,m}\end{array}\right]_{b_{1},a_{1},a_{2},...,a_{s_{1}}}\right|=\sum_{i=1}^{s_{2}}\overline{f_{i}},\\
 & = & -\overline{f_{a_{1},b_{1},a_{2},...,a_{s_{1}}}^{2}}=-\sum_{i=1}^{s_{1}}\overline{f_{i}^{2}}.\end{eqnarray*}

\end{defn}
The technique of proving the following lemma is the main technique
we are going to use for computing the $S$-pairs of elements of a
Gröbner basis.
\begin{lem}
\label{MandM}Notation as above. If \[
\mathrm{in}(m_{21}P_{a_{1},...,a_{q_{u}}}^{u})=\mathrm{in}(|M_{21}|P_{a_{1},...,a_{q_{u}}}^{u})=\mathrm{in}(\overline{P_{a_{1},..,a_{q_{u}},d_{1},...,d_{p_{21}}}^{u}})=\mathrm{in}(\overline{P_{1}^{u}})\]
 and \[
\mathrm{in}(m_{12}P_{b_{1},...,b_{q_{v}}}^{v})=\mathrm{in}(|M_{12}|P_{b_{1},...,b_{q_{v}}}^{v})=\mathrm{in}(\overline{P_{b_{1},...,b_{q_{v}},c_{1},...,c_{p_{12}}}^{v}})=\mathrm{in}(\overline{P_{1}^{v}}).\]
Furthermore $\sum_{i=2}^{p_{u}}\overline{P_{i}^{u}}$ and $\sum_{i=2}^{p_{v}}\overline{P_{i}^{v}}$
can be written as combination of elements of $G$ with leading term
smaller than $\mathrm{in}(\overline{P_{1}^{u}})$. Then the $S$-pair
of $P_{a_{1},...,a_{q_{u}}}^{u}$ and $P_{b_{1},...,b_{q_{v}}}^{v}$
has zero reminder.\end{lem}
\begin{proof}
From the definition of $M_{12}$, $M_{21}$, we have $\overline{P_{1}^{u}}=\overline{P_{1}^{v}}$.
Hence the following equation holds \begin{equation}
\overline{P_{a_{1},..,a_{q_{u}},d_{1},...,d_{p_{21}}}^{u}}-\overline{P_{b_{1},...,b_{q_{v}},c_{1},...,c_{p_{12}}}^{v}}=\sum_{i=2}^{u}\overline{P_{i}^{u}}-\sum_{i=2}^{v}\overline{P_{i}^{v}}.\label{eq:*}\end{equation}
$\overline{P_{a_{1},..,a_{q_{u}},d_{1},...,d_{p_{21}}}^{u}}$ can
be written as \[
\sum_{\{\alpha_{1},...,\alpha_{p_{21}}\}\cup\{\beta_{1},...,\beta+q_{u}\}=\{a_{1},...,a_{q_{u}},d_{1},...,d_{p_{21}}\}}|M_{\alpha_{1},...,\alpha_{p_{21}}}^{21}|P_{\beta_{1},...,\beta_{q_{u}}}^{u}\]
where $M_{\alpha_{1},...,\alpha_{p_{21}}}^{21}$ has the same rows
as $M^{21}$ with columns, $\alpha_{1},...,\alpha_{p_{21}}$ and $P_{\beta_{1},...,\beta_{q_{u}}}^{u}$
is in $G$ with columns, $\beta_{1},...,\beta_{q_{u}}$. Similarly,
$ $$\overline{P_{b_{1},...,b_{q_{v}},c_{1},...,c_{p_{12}}}^{v}}$
can be written as \[
\sum_{\{\alpha_{1},...,\alpha_{p_{21}}\}\cup\{\beta_{1},...,\beta+q_{u}\}=\{a_{1},...,a_{q_{u}},d_{1},...,d_{p_{21}}\}}|M_{\alpha_{1},...,\alpha_{p_{21}}}^{12}|P_{\beta_{1},...,\beta_{q_{u}}}^{v}.\]
 $|M_{21}|P_{a_{1},...,a_{q_{u}}}^{u}$ and $|M_{12}|P_{b_{1},...,b_{q_{v}}}^{v}$
are one of summands and their initial terms are the initial terms
of each sum. After moving everything other than $m_{21}P_{a_{1},...,a_{q_{u}}}^{u}$
and $m_{12}P_{b_{1},...,b_{q_{v}}}^{v}$ from the left-hand side of
\ref{eq:*} to the right-hand side, we obtain the equality: \[
m_{21}P_{a_{1},...,a_{q_{u}}}^{u}-m_{12}P_{b_{1},...,b_{q_{v}}}^{v}=\sum r_{i}g_{i}\]
 with $g_{i}\in G$ and $\mbox{in}(r_{i}g_{i})<\mbox{in}(m_{12}P_{b_{1},...,b_{q_{v}}}^{v})$.
\end{proof}
We are going to define some polynomials that are in the ideal $\mathcal{L}.$
Those polynomials will be part of the Gröbner basis of $\mathcal{L}$
that we are going to compute. The following definition is coming from
the $f_{a_{1},...,a_{s_{1}}}$ as defined in Lemma \ref{defineL}.
\begin{defn}
\label{f}Let $1\leq a_{1}<a_{2}<...<a_{s_{1}+k-1}\leq\mathrm{min}\{t_{1},t_{2}\}$,
and $1\leq l\leq k\leq s_{2}$, we define $f_{a_{1},...,a_{s_{1}+k-1}}^{l,k}$
as follow: \begin{eqnarray*}
 & f_{a_{1},...,a_{s_{1}+k-1}}^{l,k}:=\\
 & {\displaystyle \sum_{r=k}^{s_{2}}}(-1)^{r+1}\left|\left[\begin{array}{c}
Z^{l,k-1}\\
Z^{r,r}\\
X^{1,l-1}\\
Y^{1,r-1}\\
Y^{r+1,s_{1}}\end{array}\right]_{a_{1},...,a_{s_{1}+k-1}}\right|\\
 & +{\displaystyle \sum_{r=k}^{s_{2}}(-1)^{r+1}\sum_{u=r+1}^{s_{1}}\left|\left[\begin{array}{c}
Z^{l,k-1}\\
X^{r,r}-Y^{r,r}\\
X^{1,l-1}\\
Y^{1,r-1}\\
Y^{r+1,u-1}\\
Z^{u,u}\\
X^{u+1,s_{1}}\end{array}\right]_{a_{1},...,a_{s_{1}+k-1}}\right|}.\end{eqnarray*}
\end{defn}
\begin{lem}
\label{finL}$f_{a_{1},...,a_{s_{1}+k-1}}^{l,k}\in\mathcal{L}$.\end{lem}
\begin{proof}
We first define $p_{a_{1},...,a_{s_{1}+k-1}}^{l,k}$ as follows:\[
p_{a_{1},...,a_{s_{1}+k-1}}^{l,k}=\sum_{r=k}^{s_{2}}(-1)^{r+1}\left|\left[\begin{array}{c}
Z^{l,k-1}\\
Z^{r,r}\\
X^{1,l-1}\\
Y^{1,r-1}\\
X^{r+1,s_{1}}\end{array}\right]_{a_{1},...,a_{s_{1}+k-1}}\right|\]
where $1\leq a_{1}<a_{2}<...<a_{s_{1}+k-1}\leq\mathrm{min}\{t_{1},t_{2}\}$,
and $1\leq l\leq k\leq1$. We notice that $p_{a_{1},...,a_{s_{1}}}^{1,1}=f_{a_{1},...,a_{s_{1}}}$.
We will show $p_{a_{1},...,a_{s_{1}+k-1}}^{l,k}\in\mathcal{L}$. Since
$p_{a_{1}...a_{s_{1}+k}}^{l,k+1}=\sum_{i=1}^{s_{1}+k}(-1)^{i+1}z_{ka_{i}}f_{a_{1}...\hat{a_{i}}...a_{s_{1}+k}}^{l,k}$
and $p_{a_{1}...a_{s_{1}+l}}^{l+1,l+1}=\sum_{i=1}^{s_{1}+l}(-1)^{i+1}x_{la_{i}}p_{a_{1}...\hat{a_{i}}...a_{s_{1}+l}}^{l,l}$,
we have that the $p_{a_{1}...a_{s_{1}+k-1}}^{l,k}$'s are all in $\mathcal{L}\subset\mathcal{J}.$
By Lemma \ref{xTox_y}, we have \[
p_{a_{1},...,a_{s_{1}}}^{l,k}=\sum_{r=k}^{s_{2}}(-1)^{r+1}\left|\left[\begin{array}{c}
Z^{l,k-1}\\
Z^{r,r}\\
X^{1,l-1}\\
Y^{1,r-1}\\
X^{r+1,s_{1}}\end{array}\right]_{a_{1},...,a_{s_{1}+k-1}}\right|\]
\begin{eqnarray*}
= & \sum_{r=k}^{s_{2}}(-1)^{r+1}\left|\left[\begin{array}{c}
Z^{l,k-1}\\
Z^{r,r}\\
X^{1,l-1}\\
Y^{1,r-1}\\
Y^{r+1,s_{1}}\end{array}\right]_{a_{1},...,a_{s_{1}+k-1}}\right|\\
 & +\sum_{r=k}^{s_{2}}(-1)^{r+1}\sum_{u=r+1}^{s_{1}}\left|\left[\begin{array}{c}
Z^{l,k-1}\\
X^{r,r}-Y^{r,r}\\
X^{1,l-1}\\
Y^{1,r-1}\\
Y^{r+1,u-1}\\
Z^{u,u}\\
X^{u+1,s_{1}}\end{array}\right]_{a_{1},...,a_{s_{1}+k-1}}\right|\\
 & +\sum_{r=k}^{s_{2}}(-1)^{r+1}\sum_{u=r+1}^{s_{1}}\sum_{\{c_{1},c_{2},d_{1},...,d_{s_{1}+k-3}\}=\{a_{1},...,a_{s_{1}+k-1}\}}\\
 & \left(\pm(g_{rc_{1},uc_{2}}-g_{rc_{2},uc_{1}})\left|\left[\begin{array}{c}
Z^{l,k-1}\\
X^{1,l-1}\\
Y^{1,r-1}\\
Y^{r+1,u-1}\\
X^{u+1,s_{1}}\end{array}\right]_{d_{1},...,d_{s_{1}+k-3}}\right|\right).\end{eqnarray*}
\begin{eqnarray*}
\end{eqnarray*}
Since $p_{a_{1},...,a_{s_{1}+k-1}}^{l,k}\in\mathcal{L}$, and \begin{eqnarray*}
\sum_{r=k}^{s_{2}}(-1)^{r+1}\sum_{u=r+1}^{s_{1}}\sum_{\{c_{1},c_{2},d_{1},...,d_{s_{1}+k-3}\}=\{a_{1},...,a_{s_{1}+k-1}\}}\\
\pm(g_{rc_{1},uc_{2}}-g_{rc_{2},uc_{1}})\left|\left[\begin{array}{c}
Z^{l,k-1}\\
X^{1,l-1}\\
Y^{1,r-1}\\
Y^{r+1,u-1}\\
X^{u+1,s_{1}}\end{array}\right]_{d_{1},...,d_{s_{1}+k-3}}\right|\in\mathcal{L},\end{eqnarray*}
we have \begin{eqnarray*}
f_{a_{1},...,a_{s_{1}+k-1}}^{l,k} & = & \sum_{r=k}^{s_{2}}(-1)^{r+1}\left|\left[\begin{array}{c}
Z^{l,k-1}\\
Z^{r,r}\\
X^{1,l-1}\\
Y^{1,r-1}\\
Y^{r+1,s_{1}}\end{array}\right]_{a_{1},...,a_{s_{1}+k-1}}\right|\end{eqnarray*}
 \[
+\sum_{r=k}^{s_{2}}(-1)^{r+1}\sum_{u=r+1}^{s_{1}}\left|\left[\begin{array}{c}
Z^{l,k-1}\\
X^{r,r}-Y^{r,r}\\
X^{1,l-1}\\
Y^{1,r-1}\\
Y^{r+1,u-1}\\
Z^{u,u}\\
X^{u+1,s_{1}}\end{array}\right]_{a_{1},...,a_{s_{1}+k-1}}\right|\in\mathcal{L}.\]

\end{proof}
\medskip{}

The following definition is coming from the $S$-pairs of $X_{a_{1},...,a_{s_{1}}}$
and $g_{ij,lk}$ as defined in the Lemma \ref{defineL}.
\begin{defn}
\label{u}Let $1\leq p_{1}\leq m$, $1\leq q_{1}\leq n$, $ $$a_{s_{1}}<...<a_{j}\leq q_{1}<a_{j-1}<...<a_{1}$.
We define $U_{p_{1},q_{1},a_{1},...,a_{s_{1}}}$ as follows: \begin{eqnarray*}
 &  & U_{p_{1},q_{1},a_{s_{1}},...,a_{1}}:=z_{p_{1}q_{1}}\left|\left[\begin{array}{cccccc}
 &  & X^{1,p_{1}-1}\\
x_{p_{1}a_{s_{1}}} & ... & x_{p_{1}a_{j}} & y_{p_{1}a_{j-1}} & ... & y_{p_{1}a_{1}}\\
 &  & Y^{p_{1}+1,s_{1}}\end{array}\right]\right|\\
 &  & +\sum_{k=j+1}^{m}(x_{p_{1}q_{1}}-y_{p_{1}q_{1}})(-1)^{k+p_{1}}z_{p_{1}a_{k}}|X_{a_{1},...,\hat{a_{k}},..a_{m}}^{1,...,\hat{p_{1}},...,m}|\\
 &  & +\sum_{u=p_{1}+1}^{s_{1}}(x_{p_{1}q_{1}}-y_{p_{1}q_{1}})\left|\left[\begin{array}{cccccc}
 &  & X^{1,p_{1}-1}\\
x_{p_{1}a_{s_{1}}} & ... & x_{p_{1}a_{j}} & y_{p_{1}a_{j-1}} & ... & y_{p_{1}a_{1}}\\
 &  & Y^{p_{1}+1,u-1}\\
 &  & Z^{u,u}\\
 &  & X^{u+1,s_{1}}\end{array}\right]_{a_{1},...,a_{s_{1}}}\right|.\end{eqnarray*}

\end{defn}
\medskip{}

\begin{lem}
\label{U}$U_{p_{1}q_{1}a_{1},...,a_{s_{1}+k-1}}\in\mathcal{L}$.\end{lem}
\begin{proof}
We use Lemma \ref{xTox_y} on $|X_{a_{1},...,a_{s_{1}}}^{1,s_{1}}|$,
then we have \begin{eqnarray*}
\alpha= & z_{p_{1},q_{1}}|X_{a_{1},...,a_{s1}}|\\
= & z_{p_{1}q_{1}}\left(\left|\left[\begin{array}{cccccc}
 &  & X^{1,p_{1}-1}\\
x_{p_{1}a_{s_{1}}} & ... & x_{p_{1}a_{j}} & y_{p_{1}a_{j-1}} & ... & y_{p_{1}a_{1}}\\
 &  & Y^{p_{1}+1,s_{1}}\end{array}\right]\right|\right.\\
 & +{\displaystyle \sum_{k=j+1}^{m}}(-1)^{k+p_{1}}(x_{p_{1}a_{k}}-y_{p_{1}a_{k}})|X_{a_{1},...,\hat{a_{k}},...,a_{m}}^{1,...,\hat{p_{1}},...,m}|+{\displaystyle \sum_{u=p_{1}+1}^{s_{1}}\sum_{k=1}^{s_{1}}}\\
 & \left.(-1)^{u+k}(x_{ua_{k}}-y_{ua_{k}})\left|\left[\begin{array}{ccc}
 & X^{1,p_{1}-1}\\
x_{p_{1}a_{s_{1}}}... & x_{p_{1}a_{j}}\: y_{p_{1}a_{j-1}}... & y_{p_{1}a_{1}}\\
 & X^{p_{1}+1,u-1}\\
 & X^{u+1,s_{1}}\end{array}\right]_{a_{1},...,\widehat{a_{k}},...,a_{m}}\right|\right).\end{eqnarray*}
 Review the definition of $g_{ij,lk}$. We substitute all the monomials
that are the leading terms of $\{g_{ij,lk}\}$. The above expression
becomes\begin{eqnarray*}
 & z_{p_{1}q_{1}}\left|\left[\begin{array}{ccc}
 & X^{1,p_{1}-1}\\
x_{p_{1}a_{s_{1}}}... & x_{p_{1}a_{j}}\: y_{p_{1}a_{j-1}}... & y_{p_{1}a_{1}}\\
 & Y^{p_{1}+1,s_{1}}\end{array}\right]\right|+{\displaystyle \sum_{k=j+1}^{s_{1}}}(-1)^{p_{1}+k}g_{p_{1}q_{1},p_{1}a_{k}}|X_{a_{1},...,\hat{a_{k}},..a_{m}}^{1,...,\hat{p_{1}},...,m}|\\
 & +{\displaystyle \sum_{k=j+1}^{m}}(x_{p_{1}q_{1}}-y_{p_{1}q_{1}})(-1)^{k+p_{1}}z_{p_{1}a_{k}}|X_{a_{1},...,\hat{a_{k}},..a_{m}}^{1,...,\hat{p_{1}},...,m}|\\
 & +{\displaystyle \sum_{u=p_{1}+1}^{s_{1}}\sum_{k=1}^{m}}(-1)^{k+u}(g_{p_{1}q_{1},u,a_{k}}-g_{p_{1}a_{k},u,q_{1}})\\
 & \left(\left|\left[\begin{array}{ccc}
 & X^{1,p_{1}-1}\\
x_{p_{1}a_{s_{1}}}... & x_{p_{1}a_{j}}\: y_{p_{1}a_{j-1}}... & y_{p_{1}a_{1}}\\
 & Y^{p_{1}+1,u-1}\\
 & X^{u+1,s_{1}}\end{array}\right]_{a_{1},...,\widehat{a_{k}},...,a_{m}}\right|\right)\\
 & {\displaystyle +\sum_{u=p_{1}+1}^{s_{1}}}(x_{p_{1}q_{1}}-y_{p_{1}q_{1}})\left|\left[\begin{array}{cccccc}
 &  & X^{1,p_{1}-1}\\
x_{p_{1}a_{s_{1}}} & ... & x_{p_{1}a_{j}} & y_{p_{1}a_{j-1}} & ... & y_{p_{1}a_{1}}\\
 &  & Y^{p_{1}+1,u-1}\\
 &  & Z^{u,u}\\
 &  & X^{u+1,s_{1}}\end{array}\right]_{a_{1},...,a_{s_{1}}}\right|.\end{eqnarray*}

We define $\beta$ as follows:\begin{eqnarray*}
\beta= & {\displaystyle \sum_{k=j+1}^{s_{1}}(}-1)^{p_{1}+k}g_{p_{1}q_{1},p_{1}a_{k}}|X_{a_{1},...,\hat{a_{k}},..a_{m}}^{1,...,\hat{p_{1}},...,m}|+{\displaystyle \sum_{u=p_{1}+1}^{s_{1}}\sum_{k=1}^{m}}(-1)^{k+u}(g_{p_{1}q_{1},u,a_{k}}-g_{p_{1}a_{k},u,q_{1}})\end{eqnarray*}
\[
\left(\left|\left[\begin{array}{ccc}
 & X^{1,p_{1}-1}\\
x_{p_{1}a_{s_{1}}}... & x_{p_{1}a_{j}}\: y_{p_{1}a_{j-1}}... & y_{p_{1}a_{1}}\\
 & Y^{p_{1}+1,u-1}\\
 & X^{u+1,s_{1}}\end{array}\right]_{a_{1},...,\widehat{a_{k}},...,a_{m}}\right|\right).\]
 $\beta$ is in $\mathcal{L}$ and $\alpha$ is in $\mathcal{L},$
hence $U_{p_{1}q_{1}a_{1},...,a_{s_{1}}}=\alpha-\beta$ is in $\mathcal{L}$. 
\end{proof}
\medskip{}

The following definition is coming from the $S$-pairs of $U_{p,q,a_{1},...,a_{s_{1}}}$
as defined in the Definition \ref{u} and $Y_{a_{1},...,a_{s_{2}}}$
as defined in the Lemma \ref{defineL}. 
\begin{defn}
\label{w}Let $1\leq b_{s_{2}}<...<b_{1}\leq n$, $1\leq p_{1}\leq m$,
$1\leq q_{1}\leq n$, $a_{s_{1}}<...<a_{s_{2}+1}<a_{p_{1}}<...<a_{1}$
and $a_{p_{1}}\leq q_{1}$. Let $i$ be integer so that $1\leq i\leq p$
and $a_{s_{2}+1}<b_{s_{2}}<...<b_{i+1}<a_{p_{1}-1}\leq b_{i}$ and
$b_{l}\neq a_{p_{1}}$ for $l\geq i+1$. 

We define $M_{12}$ as follows:\[
M_{12}=z_{p_{1}q_{1}}x_{p_{1}a_{p_{1}}}\left|\left[\begin{array}{c}
X^{1,p_{1}-1}\\
Y^{s_{2}+1,s_{1}}\end{array}\right]_{a_{1},...,a_{p_{1-1}},a_{s_{2}+1},...,a_{s_{1}}}\right|.\]

We define \begin{eqnarray*}
 & W_{p_{1},q_{1},a_{1},...,a_{p1},a_{s_{2}+1},...,a_{s_{1}},b_{1},...,b_{s_{2}}}:=\\
 & M_{12}|Y_{b_{1},...,b_{s_{2}}}^{1,s_{1}}|-\left|Y_{b_{1},..,b_{i}}^{1,i}\right|{\displaystyle \sum_{\{c_{i+1},...,c_{p_{1}},d_{p_{1}+1},...,d_{s_{2}}\}=\{b_{i+1},....,b_{s_{2}}\}}}\\
 & \left|Y_{c_{i+1},..,c_{p_{1}}}^{i+1,p_{1}}\right|U_{p_{1}q_{1},a_{1},...,a_{p_{1-1}},a_{p_{1}},d_{p_{1}+1},...,d_{s_{2}},a_{s_{2}+1},...,a_{s_{1}}}.\end{eqnarray*}

\medskip{}
\end{defn}
\begin{rem}
From the way we define $W_{p_{1},q_{1},a_{1},...,a_{p_{1}},a_{s_{2}+1},...,a_{s_{1}},b_{1},...,b_{s_{1}}}$,
it is in $\mathcal{L}$. Notice that all the submatrices $|Y_{d_{p_{1}+1},...,d_{s_{2}}}^{p_{1},s_{2}}|$
of $|Y_{b_{1},...,b_{s_{2}}}^{1,s_{2}}|$ such that $a_{s_{2}+1}<d_{s_{2}}<...<d_{p_{1}+1}<a_{p_{1}-1}$
are cancelled. Hence the leading term is \[
\mbox{in}(M_{12}|Y_{b_{1},...,b_{i-1}}^{1,i-1}||Y_{b_{i+1},...,b_{p_{1+1}}}^{i,p_{1}}||Y_{b_{s_{2}},...,b_{p_{1}+2},b_{i}}^{p_{1}+1,s_{2}}|).\]

\medskip{}

\end{rem}
The following definition is coming from the $S$-pairs of $W_{p,q,a_{1},...,a_{p},a_{s_{2}+1},...,a_{s_{1}},b_{1},...,b_{s_{1}}}$
as defined in the Definition \ref{w} and $U_{p,q,a_{1},...,a_{s_{1}}}$
as defined in the Definition \ref{u}.
\begin{defn}
\label{wp}Let $1\leq p_{1}\leq m$, $1\leq q_{1}\leq n$, $v=p_{1}+1,...,s_{2}-1$,
$1\leq a_{s_{1}}<...<a_{s_{2}+1}<a_{p_{1}}<...<a_{1}\leq t_{1}$ and
$a_{p_{1}}\leq q_{1}$. Let $i$ be integer so that $1\leq i\leq p$
and let $a_{s_{2}+1}<b_{s_{2}}<...<b_{v+2}<b_{v}^{'}<...<b_{p_{1}+1}^{'}<b_{p_{1}}<...<b_{i+1}<a_{p_{1}-1}\leq b_{p_{1}+1}$
and $b_{l}^{'}\neq a_{p_{1}}$ for $l\geq i+1$ and $b_{v-1}^{'}\leq b_{v+1}$.
Let $a_{s_{1}}<....<a_{s_{2}+1}<b_{s_{2}}<...<b_{v+2}<b_{v+1}<b_{v}<b_{v-1}<...<b_{p_{1}+2}<a_{p_{1}}<a_{p_{1}-1}\leq b_{p_{1}+1}$,
and $b_{r}^{'}\leq b_{r+2}<b_{r+1}$ for $r=p_{1},...,v-2$. 

We define \begin{eqnarray*}
 & W_{p_{1},q_{1},a_{1},...,a_{p1},a_{s_{2}+1},...,a_{s_{1}},b_{1},...,b_{p_{1}+1},b_{p_{1}+2},b_{p_{1}+3}.,b_{s_{2}},b_{p_{1}+1}^{'},b_{p_{1}+2}^{'},...,b_{v}^{'}}^{p_{1}+1,v}:=\\
 & y_{v-1,b_{v-1}^{'}}y_{v,b_{v}}W_{p_{1},q_{1},a_{1},...,a_{p1},a_{s_{2}+1},...,a_{s_{1}},b_{1},...,b_{v-1},b_{v}^{'},b_{v+1},...,b_{s_{2}},b_{p_{1}+1}^{'},b_{p_{1}+2}^{'},...,b_{v-2}^{'}}^{p_{1}+1,v-2}\\
 & -y_{v,b_{v}^{'}}W_{p_{1},q_{1},a_{1},...,a_{p1},a_{s_{2}+1},...,a_{s_{1}},b_{1},...,b_{s_{2}},b_{p_{1}+1}^{'},b_{p_{1}+2}^{'},...,b_{v-1}^{'}}^{p_{1}+1,v-1}.\end{eqnarray*}
Here \[
W_{p_{1},q_{1},a_{1},...,a_{p_{1}},a_{s_{2}+1},...,a_{s_{1}},b_{1},...,b_{s_{2}}}^{p_{1}+1,p_{1}-1}=U_{p_{1},q_{1},a_{1},...,a_{s_{1}}}\]
 and \[
W_{p_{1}q_{1},a_{1},...,a_{p_{1}},a_{s_{2}+1},...,a_{s_{1}},b_{1},...,b_{s_{2}}}^{p_{1}+1,p_{1}}=W_{p_{1}q_{1},a_{1},..,a_{p_{1}},a_{s_{2}+1},...,a_{s_{1}},b_{1},...,b_{s_{2}}}.\]

\end{defn}
\medskip{}

\begin{rem}
From the way we define $W_{p_{1},q_{1},a_{1},...,a_{p_{1}},a_{s_{2}+1},...,a_{s_{1}},b_{1},...,b_{s_{2}},b_{p_{1}+1}^{'},...,b_{v}^{'}}^{p_{1}+1,v}$,
it is clear that it belongs to $\mathcal{L}$. Notice it has leading
term \begin{eqnarray*}
\mbox{in}(z_{p_{1}q_{1}}\left|\left[\begin{array}{c}
X^{1,p_{1}}\\
Y^{s_{2}+1,s_{1}}\end{array}\right]_{a_{1},...,a_{p_{1}},a_{s_{2}+1},...,a_{s_{1}}}\right||Y_{b_{1},...,b_{i-1}}^{1,i-1}||Y_{b_{i+1},...,b_{p_{1+1}}}^{i,p_{1}}|\\
y_{p_{1}+1,b_{i}}y_{p_{1}+1,b_{p_{1}+1}}y_{p_{1}+2,b_{p_{1}+2}}y_{p_{1}+2,b_{p_{1}+2}^{'}}...y_{v-1,b_{v-1}^{'}}y_{v-1,b_{v-1}^{'}}y_{vb_{v}}y_{v,b_{v}^{'}}\\
|Y_{b_{s_{2}},...,b_{v+2},b_{v+1}}^{l+1,s_{2}}|).\end{eqnarray*}

\medskip{}

\end{rem}
The following definition is coming from the $S$-pairs of $f_{a_{1},...,a_{s_{1}+k-1}}^{l,k}$
as defined in the Definition \ref{f} and $Y_{a_{1},...,a_{s_{2}}}$
as defined in the Definition \ref{defineL}.
\begin{defn}
\label{v}Let $b_{s_{1}}<...<b_{1}$, and $1\leq p_{l}<...<p_{k}<b_{s_{1}}<...<b_{s_{2}+1}<c_{s_{2}}<...<c_{k+1}<b_{k-1}<...<b_{1}<a_{k-1}<...<a_{1}\leq t_{1}$. 

Let \[
M_{12}=\left|\left[\begin{array}{c}
Z^{l,k}\\
X^{1,k-1}\\
Y^{s_{2}+1,s_{1}}\end{array}\right]_{p_{l},..,p_{k},a_{1},...,a_{k-1},b_{s_{2}+1},...,b_{s_{1}}}\right|.\]

We define \begin{eqnarray*}
 & V_{p_{l},...,p_{k},a_{1},...,a_{k-1},b_{1},...,b_{s_{1}}}:=\\
 & M_{12}|Y_{b_{1},...,b_{s_{2}}}^{1,s_{2}}|{\displaystyle -\sum_{\{e_{k},c_{k+1},...,c_{s_{2}}\}=\{b_{k},....,b_{s_{2}}\}}}\\
 & \pm y_{ke_{k}}f_{p_{l},...,p_{k},a_{1},...,a_{k-1},b_{1},...,b_{k-1},c_{k+1},...,c_{s_{2}},b_{s_{2}+1},...,b_{s_{1}}}^{l,k}.\end{eqnarray*}
\end{defn}
\begin{rem}
From the way we define $V_{p_{l},...,p_{k},a_{1},...,a_{k-1},b_{1},...,b_{s_{1}}}$,
it is in $\mathcal{L}$. Notice the submatrices $|Y_{c_{k+1},...,c_{s_{2}}}^{k+1,s_{2}}|$
of $|Y_{b_{1},...,b_{s_{2}}}^{1,s_{2}}|$ such that $b_{s_{2}+1}<c_{s_{2}}<...<c_{k+1}<b_{k-1}$
are cancelled. Hence the leading term of $V_{p_{l},...,p_{k},a_{1},...,a_{k-1},b_{1},...,b_{s_{1}}}$
is \[
\mbox{in}(M_{12}|Y_{b_{1},..,b_{k-2}}^{1,k-2}|y_{k-1,b_{k}}|Y_{b_{s_{2}},...,b_{k+1},b_{k-1}}^{k,s_{2}}|).\]

\end{rem}
\medskip{}

The following definition is coming from the $S$-pairs of elements
in \linebreak{}
$\{V_{p_{1},..,p_{k},a_{1},...,a_{k-1},b_{1},...,b_{s_{1}}}\}$
as defined in the Definition \ref{v}.
\begin{defn}
\label{vk}Let $1\leq l\leq k\leq s_{2}$ and $1\leq p_{l}<...<p_{k}<b_{s_{1}}<...<b_{s_{2}+1}<...<b_{k+1}<b_{k-1}<b_{k}<b_{k-2}<...<b_{1}<a_{l-1}<...<a_{1}\leq t_{1}$.
Let $w=k,...,s_{2}-1$ and $1\leq b_{s_{2}}<...<b_{w+2}<b_{w}^{'}<b_{w-1}^{'}<...<b_{k}^{'}<b_{k-1}<b_{k-2}...<b_{1}\leq t_{2}$
and $b_{w-1}^{'}\leq b_{w+1}$, and $b_{r}^{'}\leq b_{r+2}<b_{r+1}$
for $r=k,...,l-2$. 

We define \begin{eqnarray*}
 & V_{p_{l},...,p_{k},a_{1},...,a_{k-1},b_{1},...,b_{s_{1}},b_{k}^{'},b_{k+1}^{'},...,b_{w}^{'}}^{k,w}:=\\
 & y_{w-1,b_{w-1}}y_{w,b_{w}}V_{p_{l},...,p_{k},a_{1},...,a_{k-1},b_{1},..,b_{w}^{'},...,b_{s_{1}},b_{k}^{'},b_{k+1}^{'},...,b_{w-2}^{'}}^{k,w-2}\\
 & -y_{wb_{w}^{'}}V_{p_{l},...,p_{k},a_{1},...,a_{k-1},b_{1},...,b_{s_{1}},b_{k}^{'},b_{k+1}^{'},...,b_{w-1}^{'}}^{k,w-1} & .\end{eqnarray*}
Here $ $$V_{p_{l},...,p_{k},a_{1},...,a_{k-1},b_{1},...,b_{s_{1}}}^{k,k-2}=V_{p_{l},...,p_{k},a_{1},...,a_{k-1},b_{1},...,b_{s_{1}}}^{k,k-1}=V_{p_{l},...,p_{k},a_{1},...,a_{k-1},b_{1},...,b_{s_{1}}}$.\end{defn}
\begin{rem}
From the way we define $V_{p_{l},...,p_{k},a_{1},...,a_{k-1},b_{1},...,b_{s_{1}},b_{k}^{'},b_{k+1}^{'},...,b_{l}^{'}}^{k,w}$,
it is in $\mathcal{L}$. It has leading term \[
\mbox{in}(M_{12}|Y_{b_{1},..,b_{k-2}}^{1,k-2}|y_{k-1,b_{k-1}}y_{kb_{k}}y_{kb_{k}^{'}}...y_{lb_{w}}y_{lb_{w}^{'}}|Y_{b_{s_{2}},...,b_{w+1}}^{l+1,s_{2}}|).\]

\end{rem}
\medskip{}

The following definition is coming from the $S$-pairs of $g_{ij,lk}$
as defined in the Definition \ref{defineL} and $f_{a_{1},...,a_{s_{1}+k-1}}^{l,k}$
as defined in the Definition \ref{f}.
\begin{defn}
\label{h}Let $1\leq l\leq k\leq s_{2}$, $1\leq q\leq n$, $1\leq a_{s_{1}+k-1}<....<a_{1}\leq t_{1}$,
$a_{s_{1}+k-1}<q$, $a_{j+1}\leq q<a_{j}$ for some $j=l-1,...,s_{1}+k-3$.
Let $\overline{f_{a_{1},...,\hat{a_{c}},...,a_{s_{1+k-1}}}^{l,k,x_{l-1}}}$
be the determinant of matrices that coming from deleting row $x_{l-1}$
and column $a_{c}$. We define $H_{a_{1},....,a_{s_{1}+k-1}}^{l,k,q}$
as follows \begin{eqnarray*}
H_{a_{1},...,a_{s_{1}+k-1}}^{l,k,q} & = & z_{l-1,q}f_{a_{1},...,a_{s_{1}+k-1}}^{l,k}-\sum_{c=k}^{j}(-1)^{k+c}g_{l-1,q,l-1,a_{c}}\overline{f_{a_{1},...,\hat{a_{c}},...,a_{s_{1+k-1}}}^{l,k,x_{l-1}}}.\end{eqnarray*}
\end{defn}
\begin{rem}
It is clear that $H_{a_{1},...,a_{s_{1}+k-1}}^{l,k,q}$ is in $\mathcal{L}$
from the way we define it. Notice in the row $x_{l}$ of $f_{a_{1},...,a_{s_{1}+k-1}}^{l,k}$,
the $x_{l,a_{c}}$ are cancelled by the $g_{l-1,q,l-1a_{c}}$. Hence
the leading term of $H_{a_{1},...,a_{s_{1}+k-1}}^{l,k,q}$ is\\
 \[
z_{l-1,q}\mbox{in}\left(\left|Z_{p_{l},...,p_{k}}^{l,k}\right|x_{l-1,a_{j+1}}\left|X_{a_{1},...,a_{l-2}}^{1,l-2}\right|\left|\left[\begin{array}{c}
Y^{1,k-1}\\
Y^{k+1,s_{1}}\end{array}\right]_{b_{1},...,b_{k-1},b_{k+1},...,b_{s_{1}}}\right|\right).\]
Here $p_{i}\neq a_{j+1}$, $b_{i}\neq a_{j+1}$ for all $i$. \medskip{}

\end{rem}
The following definition is coming from the $S$-pairs of $H_{a_{1},...,a_{s_{1}+k-1}}^{l,k,q}$
as defined in the Definition \ref{h} and $Y_{a_{1},...,a_{s_{2}}}$
as defined in the Definition \ref{defineL}.
\begin{defn}
\label{i}Let $1\leq l\leq k\leq s_{2}$, $1\leq q\leq n$, $1\leq a_{s_{1}+k-1}<....<a_{1}\leq t_{1}$,
$a_{s_{1}+k-1}<q$, $a_{j+1}\leq q<a_{j}$ for some $j=l-1,...,s_{1}+k-3$.
Let $a_{l+s_{2}-1}<b_{s_{2}}<...<b_{k}<a_{l-1+k-1}=b_{k-1}<....<a_{l-1+1}=b_{1}$. 

Let \[
M=z_{l-1,q}x_{l-1,a_{j+1}}\left|\left[\begin{array}{c}
Z^{l,k}\\
X^{1,l-2}\\
Y^{s_{2}+1,s_{1}}\end{array}\right]_{a_{s_{1}+k-1},...,a_{l+s_{2}-1},a_{l-1},...,a_{1}}\right|.\]

We define \begin{eqnarray*}
 &  & I_{a_{s_{1}+k-1},...,a_{l+s_{2}-1},a_{l-1},...,a_{1},b_{1},...,b_{s_{2}}}^{l,k,q}:=\\
 &  & M|Y_{b_{1},...,b_{s_{2}}}^{1,s_{2}}|-\sum_{\{e_{k},c_{k+1},...,c_{s_{2}}\}=\{b_{k},....,b_{s_{2}}\}}\\
 &  & \pm y_{ke_{k}}H_{a_{s_{1}+k-1},...,a_{l+s_{2}-1},c_{s_{2}},...,c_{k+1},b_{k-1},...,b_{1},a_{l-1},...,a_{1}}^{l,k,q}.\end{eqnarray*}
\end{defn}
\begin{rem}
It is clear that $I_{a_{s_{1}+k-1},...,a_{l+s_{2}-1},a_{l-1},...,a_{1},b_{1},...,b_{s_{2}}}^{l,k,q}$
is in $\mathcal{L}$ from the way we define it. Notice that the submatrices
$|Y_{c_{k+1},...,c_{s_{2}}}^{k+1,s_{2}}|$ of $|Y_{b_{1},...,b_{s_{2}}}^{1,s_{2}}|$
with $a_{l+s_{2}-1}<c_{s_{2}}<...<c_{k+1}<b_{k-1}$ are cancelled
by $H^{l,k,q}$'s, hence the leading term of \\
$I_{a_{s_{1}+k-1},...,a_{l+s_{2}-1},a_{l-1},...,a_{1},b_{1},...,b_{s_{2}}}^{l,k,q}$
is\\
 \begin{eqnarray*}
\mbox{in}\left(z_{l-1,q}x_{l-1,a_{j+1}}\left|\left[\begin{array}{c}
Z^{l,k}\\
X^{1,l-2}\\
Y^{s_{2}+1,s_{1}}\end{array}\right]_{a_{s_{1}+k-1},...,a_{l+s_{2}-1},a_{l-1},...,a_{1}}\right|\right.\\
\left.y_{k-1,b_{k}}\left|\left[\begin{array}{c}
Y^{1,k-2}\\
Y^{k,s_{2}}\end{array}\right]_{b_{s_{2}},...,b_{k+1},b_{k-1},b_{k-2}...b_{1}}\right|\right).\end{eqnarray*}

\end{rem}
\medskip{}

The following definition is coming from the $S$-pairs of elements
of \linebreak{}
$\{I_{q_{l},...,q_{k},b_{s_{1}},...,b_{s_{2}},...,b_{1},a_{l-1},...,a_{1}}^{l,k,q}\}$.
\begin{defn}
\label{ik}Let $1\leq l\leq k\leq s_{2}$, $1\leq q\leq n$, $k\leq w\leq s_{2}-1$,
$1\leq q_{l}<...<q_{k}<b_{s_{1}}<...<b_{s_{2}}<...<b_{k+1}<b_{k-1}<b_{k}<b_{k-2}<...<b_{1}<a_{l-2}<...<a_{1}\leq t_{1}$,
$q_{l}<q$, $q_{l}<a_{l-1}\leq q$ . Let $w=k,...,s_{2}-1$ and $1\leq b_{s_{2}}<...<b_{w+2}<b_{w}^{'}<b_{w-1}^{'}<...<b_{k}^{'}<b_{k-1}<b_{k-2}...<b_{1}\leq t_{2}$
and $b_{w-1}^{'}\leq b_{w+1}$, and $b_{r}^{'}\leq b_{r+2}<b_{r+1}$
for $r=k,...,l-2$. 

We define \begin{eqnarray*}
 & ^{k,w}I_{q_{l},...,q_{k},b_{s_{1}},...,b_{s_{2}},...,b_{1},a_{l-1},...,a_{1},b_{k}^{'},...,b_{w}^{'}}^{l,k,q}:=\\
 & y_{w-1,b_{w-1}^{'}}y_{wb_{w}}{}^{k,w-2}I_{q_{l},...,q_{k},b_{s_{1}},...,b_{s_{2}},...,b_{w}^{'},...,b_{1},a_{l-1},...,a_{1},b_{k}^{'},...,b_{w-2}^{'}}^{l,k,q}\\
 & -y_{wb_{w}^{'}}{}^{k,w-1}I_{q_{l},...,q_{k},b_{s_{1}},...,b_{s_{2}},...,b_{1},a_{l-1},...,a_{1},b_{k}^{'},...,b_{w-1}^{'}}^{l,k,q}.\end{eqnarray*}
Here \begin{eqnarray*}
^{k,k-2}I_{q_{l},...,q_{k},b_{s_{1}},...,b_{s_{2}},...,b_{1},a_{l-1},...,a_{1}}^{l,k,q} & ={}^{k,k-1}I_{q_{l},...,q_{k},b_{s_{1}},...,b_{s_{2}},...,b_{1},a_{l-1},...,a_{1}}^{l,k,q}\\
=I_{q_{l},...,q_{k},b_{s_{1}},...,b_{s_{2}},...,b_{1},a_{l-1},...,a_{1}}^{l,k,q} &  & .\end{eqnarray*}
\end{defn}
\begin{rem}
From the way we define $^{k,w}H_{a_{s_{1}+k1-},...,a_{l-1+s_{2}}b_{s_{2}},...,b_{1},a_{l-1},...,a_{1},b_{k}^{'},...,b_{w}^{'}}^{l,k,q}$,
it is in $\mathcal{L}$. The leading term of $^{k,w}H_{a_{s_{1}+k1-},...,a_{l-1+s_{2}}b_{s_{2}},...,b_{1},a_{l-1},...,a_{1},b_{k}^{'},...,b_{w}^{'}}^{l,k,q}$
is \begin{eqnarray*}
\mbox{in}\left(z_{l-1,q}x_{l-1,a_{j+1}}\left|\left[\begin{array}{c}
Z^{l,k}\\
X^{1,l-2}\\
Y^{s_{2}+1,s_{1}}\end{array}\right]_{a_{s_{1}+k-1},...,a_{l+s_{2}-1},a_{l-1},...,a_{1}}\right|\right.\\
\left.y_{k-1,b_{k-1}}y_{kb_{k}}y_{kb_{k}^{'}}...y_{wb_{w}}y_{wb_{w}^{'}}\left|\left[\begin{array}{c}
Y^{1,k-2}\\
Y^{w+1,s_{2}}\end{array}\right]_{b_{s_{2}},...,,b_{w+1},b_{k-2}...b_{1}}\right|\right).\end{eqnarray*}

\end{rem}
\medskip{}

We are now ready to find the Gröbner basis of $\mathcal{L}$.

\medskip{}

\begin{thm}
\label{GB}Use the notation of Definition \ref{f}, \ref{u}, \ref{w},
\ref{wp}, \ref{v}, \ref{vk}, \ref{h}, \ref{i}, \ref{ik} and
let $\mathcal{G}:=\{|X_{a_{1},..,a_{s_{1}}}^{1,s_{1}}|$, $|Y_{b_{1},...,b_{s_{2}}}^{1,s_{2}}|$,
$g_{p_{1}q_{1},p_{2}q_{2}}$, $f_{a_{1},...,a_{s_{1}+k-1}}^{l,k}$,
$U_{p_{1},q_{1},a_{1},...,a_{s_{1}}}$ , \\
$W_{p,q,a_{1},...,a_{p},a_{s_{2}+1},...,a_{s_{1}},b_{1},...,b_{s_{2}}}$,
\\
$W_{p_{1},q_{1},a_{1},...,a_{p1},a_{s_{2}+1},...,a_{s_{1}},b_{1},...,b_{p_{1}+1},b_{p_{1}+2},b_{p_{1}+3}.,b_{s_{2}},b_{p_{1}+1}^{'},b_{p_{1}+2}^{'},...,b_{v}^{'}}^{p_{1}+1,v}$,
\\
$V_{p_{l},...,p_{k},a_{1},...,a_{k-1},b_{1},...,b_{s_{1}}}$, $V_{p_{l},...,p_{k},a_{1},...,a_{k-1},b_{1},...,b_{s_{1}},b_{k}^{'},b_{k+1}^{'},...,b_{w}^{'}}^{k,w}$,
$H_{a_{1},...,a_{s_{1}+k-1}}^{l,k,q}$ ,\\
 \textup{$I_{a_{s_{1}+k-1},...,a_{l+s_{2}-1},a_{l-1},...,a_{1},b_{1},...,b_{s_{2}}}^{l,k,q}$,}
$^{k,w}I_{q_{l},...,q_{k},b_{s_{1}},...,b_{s_{2}},...,b_{1},a_{l-1},...,a_{1},b_{k}^{'},...,b_{w}^{'}}^{l,k,q}\}$. 

The $\mathcal{G}$ is a Gröbner basis of $\mathcal{L}$ with respect
to the lexicographic term order and the variables ordered by $z_{ij}>x_{lk}>y_{pq}$
for any $i,j,l,k,p,q$ and $x_{ij}<x_{lk}$, $y_{ij}<y_{lk}$ if $i>l$
or $i=l$ and $j<k$ and $z_{ij}<z_{lk}$ if $i>l$ or if $i=l$ and
$j>k$.
\end{thm}
We break up the proof of the above theorem into a sequence of lemmas
when we treat $S$-pairs between elements of $\mathcal{G}$. We only
have to compute the $S$-pairs of elements whose leading terms are
not relative prime. In each lemma, we show $h_{P,Q}=0$ for some $P,Q$
in $\mathcal{G}$. We define a order on pair $(i,j)$ with $1\leq i\leq m$,
$1\leq j\leq n$. We say $(i,j)>(l,k)$ if $i<l$ or $i=l$ and $j<k$.
This is a total order. 
\begin{lem}
$h_{P,Q}=0$ when $P$ and $Q$ are in the same group of $\mathcal{G}$.\end{lem}
\begin{proof}
We use notation in the Definition \ref{TopM}. Notice that $m_{12}$
and $m_{21}$ have the same row indices, $P_{a_{1},...,a_{q_{u}}}^{u}$
and $Q_{b_{1},...,b_{q_{v}}}^{v}$ have the same number of columns,
i.e. $q_{u}=q_{v}$ hence \[
\overline{P_{a_{1},..,a_{q_{u}},d_{1},...,d_{p_{21}}}^{u}}=\overline{Q_{b_{1},...,b_{q_{v}},c_{1},...,c_{p_{12}}}^{v}}.\]
 Also $ $$\mbox{in}(m_{12}P_{a_{1},...,a_{q_{u}}}^{u})=\mbox{in}(m_{21}Q_{b_{1},...,b_{q_{u}}}^{v})$
are indeed the leading terms of \\
$\overline{P_{a_{1},..,a_{q_{u}},d_{1},...,d_{p_{21}}}^{u}}$.
The first matrix of $\overline{P_{a_{1},..,a_{q_{u}},d_{1},...,d_{p_{21}}}^{u}}$
has determinant equal to zero, since it has repeated row. Hence we
have $m_{12}P_{a_{1},...,a_{q_{u}}}^{u}$ and $ $ $m_{21}P_{b_{1},...,b_{q_{u}}}^{v}$
having different signs in the sum. Except $P_{a_{1},...,a_{q_{u}}}^{u}=f_{a_{1},...,a_{q_{u}}}^{l,k}$
and $Q_{b_{1},...,b_{q_{v}}}^{v}=f_{a_{1},...,b_{q_{u}}}^{l,k}$with
$a_{i}=b_{i}$ for $i\neq k-l+1$ and $a_{k-l+1}\neq b_{k-l+1}$,
each summand of all possible cases will have either repeated row,
or all the rows, $y_{1},...,y_{s_{2}}$ or rows of lemma \ref{topx}.
Hence they give \[
\sum_{i=2}^{u}\overline{P_{i}}\in\mathcal{G}.\]
For the remaining case, $ $\[
\sum_{i=2}^{u}\overline{P_{i}}=f_{a_{1},...,a_{k-l+1},b_{k-l+1},a_{k-l+2},...,a_{q_{u}}}^{l,k+1}\]
 from the proof of Lemma \ref{finL}. Hence the following is true:
\[
\overline{P_{a_{1},..,a_{q_{u}},d_{1},...,d_{p_{21}}}^{u}}=\sum_{i=2}^{u}\overline{P_{i}}\in\mathcal{G}.\]
After moving everything other than $m_{12}P_{a_{1},...,a_{q_{u}}}^{u}$
and $m_{21}Q_{b_{1},...,b_{q_{v}}}^{v}$ from the left hand side to
the right hand side, we obtain the $S$-pair and it becomes: \[
m_{12}P_{a_{1},...,a_{q_{u}}}^{u}-m_{21}Q_{b_{1},...,b_{q_{v}}}^{v}=\sum r_{i}f_{i}\]
with $\mbox{in}(r_{i}f_{i})<\mbox{in}(m_{12}P_{a_{1},...,a_{q_{u}}}^{u})$
and $f_{i}\in\mathcal{G}$.\end{proof}
\begin{lem}
$h_{P,Q}=0$ when $P\in\{|X_{a_{1},...,a_{s_{1}}}^{1,s_{1}}|\}$ in
$\mathcal{G}$.\end{lem}
\begin{proof}
As the notation in Definition \ref{TopM}, we look at $\sum_{i=2}^{u}\overline{P_{i}}$.
For most of cases, $\sum_{i=2}^{u}\overline{P_{i}}=0$, since each
summand has repeated rows $x_{j}$ for some $j=1,...,s_{1}$. In some
other cases, we have either rows, $y_{i},\: x_{i}$ and $z_{i}$ in
each matrix then Lemma \ref{g_ij} can be applied. Or the part of
the sum has sum as Lemma \ref{topx} then deduce that it is in $(\{|Y_{b_{1},...,b_{s_{2}}}^{1,s_{2}}|\})+(\{g_{ij,lk}\})$.
Similarly, $\sum_{i=2}^{u}\overline{Q_{i}}\in(\{|Y_{b_{1},...,b_{s_{2}}}^{1,s_{2}}|\})+(\{g_{ij,lk}\})$,
hence Lemma \ref{MandM} applies.\end{proof}
\begin{lem}
$h_{P,Q}=0$ when $P\in\{|Y_{b_{1},...,b_{s_{2}}}^{1,s_{2}}|\}$ in
$\mathcal{G}$.\end{lem}
\begin{proof}
The computation of $S$-pair between $f_{a_{1},...,a_{s_{1}+k-1}}^{l,k}$
and $|Y_{b_{1},...,b_{s_{2}}}^{1,s_{2}}|$ gives us $V_{p_{l},...,p_{k},a_{1},...,a_{k-1},b_{1},...,b_{s_{1}}}$
as in Definition \ref{v}. The $S$-pair between \\
$V_{p_{l},...,p_{k},a_{1},...,a_{k-1},b_{1},...,b_{s_{1}}}$ and
$|Y_{b_{1},...,b_{s_{2}}}^{1,s_{1}}|$ gives us $V_{p_{l},...,p_{k},a_{1},...,a_{k-1},b_{1},...,b_{s_{1}},b_{k}^{'},b_{k+1}^{'},...,b_{w}^{'}}^{k,w}$
as Definition \ref{vk}. The $S$-pair between $V_{p_{l},...,p_{k},a_{1},...,a_{k-1},b_{1},...,b_{s_{1}},b_{k}^{'},b_{k+1}^{'},...,b_{w}^{'}}^{k,w}$
and $|Y_{b_{1},...b_{s_{2}}}^{1,s_{2}}|$ gives $V_{p_{l},...,p_{k},a_{1},...,a_{k-1},b_{1},...,b_{s_{1}},b_{k}^{'},b_{k+1}^{'},...,b_{w+1}^{'}}^{k,w+1}$.
Similarly the computation of $S$-pair between $ $$H_{a_{1},...,a_{s_{1}+k-1}}^{l,k,q}$
and $|Y_{b_{1},...,b_{s_{2}}}^{1,s_{2}}|$ gives $ $$I_{a_{s_{1}+k-1},...,a_{l+s_{2}-1},a_{l-1},...,a_{1},b_{1},...,b_{s_{2}}}^{l,k,q}$
as Definition \ref{i} and the $S$-pair between $ $$I_{a_{s_{1}+k-1},...,a_{l+s_{2}-1},a_{l-1},...,a_{1},b_{1},...,b_{s_{2}}}^{l,k,q}$
and $|Y_{b_{1},...,b_{s_{2}}}^{1,s_{2}}|$ gives $ $$^{k,w}I_{q_{l},...,q_{k},b_{s_{1}},...,b_{s_{2}},...,b_{1},a_{l-1},...,a_{1},b_{k}^{'},...,b_{w}^{'}}^{l,k,q}$.
Also the $S$-pair between $ $$U_{p_{1},q_{1},a_{1},...,a_{s_{1}}}$
and $|Y_{b_{1},...,b_{s_{2}}}^{1,s_{2}}|$ gives $ $$W_{p,q,a_{1},...,a_{p},a_{s_{2}+1},...,a_{s_{1}},b_{1},...,b_{s_{2}}}$
as Definition \ref{w} and the $S$-pair between $ $$W_{p,q,a_{1},...,a_{p},a_{s_{2}+1},...,a_{s_{1}},b_{1},...,b_{s_{2}}}$
and $|Y_{b_{1},...,b_{s_{2}}}^{1,s_{2}}|$ gives \\
$W_{p_{1},q_{1},a_{1},...,a_{p1},a_{s_{2}+1},...,a_{s_{1}},b_{1},...,b_{p_{1}+1},b_{p_{1}+2},b_{p_{1}+3}.,b_{s_{2}},b_{p_{1}+1}^{'},b_{p_{1}+2}^{'},...,b_{v}^{'}}^{p_{1}+1,v}$.\end{proof}
\begin{lem}
$h_{P,Q}=0$ when $P\in\{g_{ij,lk}\}$ in $\mathcal{G}$.\end{lem}
\begin{proof}
If $Q\in\{g_{ij,lk}\}$, we have $Q=z_{p_{1}q_{1}}(x_{p_{2}q_{2}}-y_{p_{2}q_{2}})-z_{p_{2}q_{2}}(x_{p_{1}q_{1}}-y_{p_{1}q_{1}})$
and $P=g_{ij,lk}=z_{ij}(x_{lk}-y_{lk})-z_{lk}(x_{ij}-y_{ij})$. It's
sufficient to consider either $(p_{1},q_{1})=(i,j)$ or $(p_{2},q_{2})=(l,k)$.
For the first case, \begin{eqnarray*}
(x_{lk}-y_{lk})Q-(x_{p_{2}q_{2}}-y_{p_{2}q_{2}})P & =(x_{ij}-y_{ij}) & g_{lk,p_{2}q_{2}}.\end{eqnarray*}
For the second case, \[
z_{ij}Q-z_{p_{1}q_{1}}P=z_{p_{2}q_{2}}g_{p_{1}q_{1},ij}.\]
Notice that $P=g_{ij,lk}=z_{i,j}(x_{l,k}-y_{l,k})-z_{l,k}(x_{i,j}-y_{i,j})$
with $(i,j)>(l,k)$. If $Q\in\{|X_{a_{1},...,a_{s_{1}}}^{1,s_{1}}|\}$,
the computing of $S$-pair of $P$ and $Q$ is similar to Lemma \ref{U}.
And it gives $U_{pqa_{1},...,a_{s_{1}}}$ as Definition $ $\ref{u}.
If $Q\in\{f_{a_{1},...,a_{s_{1+k-1}}}^{l,k}\}$, the computing of
$S$-pair of $P$ and $Q$ will give us $H_{a_{1},...,a_{s_{1}+k-1}}^{l,k,q}$
as Definition \ref{h} when $P=g_{ij,lk}$ and $i=l-1$. Otherwise
$GCD(\mbox{in}(P),\:\mbox{in}(Q))=z_{i,j}$ with $i\in\{l,l+1,...,k\}$
or $GCD(\mbox{in}(P),\:\mbox{in}(Q))=x_{l,k}$ with $i<l-1$. For
$GCD(\mbox{in}(P),\:\mbox{in}(Q))=z_{i,j}$ with $i\in\{l,l+1,...,k\}$,
the computation of $S$-pair gives us $f^{l+1,k}$. For $GCD(\mbox{in}(P),\:\mbox{in}(Q))=x_{l,k}$,
the computation of $S$-pair gives us repeated row, $y_{l}$, in every
matrix of $Q$ and this makes the determinant equal to zero. For all
other cases, $Q\in\mathcal{G}$, they come from the $S$-pair of $P\in\{g_{ij,lk}\}$
and $|X_{a_{1},...,a_{s_{1}}}^{1,s_{1}}|$ or $f_{a_{1},...,a_{s_{1}+k-1}}^{l,k}$.
Hence the computations of $S$-pair are very similar.\end{proof}
\begin{lem}
$h_{P,Q}=0$ when $P=f_{a_{1},...,a_{s_{1}+k_{1}-1}}^{l_{1},k_{1}}$
and $Q=f_{b_{1},...,b_{s_{1}+k_{2}-1}}^{l_{2},k_{2}}$ and $l_{1}\neq l_{2}$
or $k_{1}\neq k_{2}$ in $\mathcal{G}$.\end{lem}
\begin{proof}
We prove this part in two cases: (a) $k_{1}\neq k_{2}$, (b) $l_{1}\neq l_{2}$.

In case (a), without lost of generality, let $k_{1}>k_{2}$. Then
the matrix appears in the first summand of $f_{a_{1},..,a_{s_{1}+k_{1}-1}}^{l_{1}k_{1}}$
has row $y_{k_{2}}$ without row $y_{k_{1}}$, and the matrix that
appears in the first summand of $f_{b_{1},...,b_{s_{1}+k_{2}-1}}^{l_{2}k_{2}}$
has row $y_{k_{1}}$without $y_{k_{2}}$. Consider $m_{12}$, $m_{21}$,
$M_{12}$ and $M_{21}$ as defined in Definition \ref{M_ij}. Assume
$M_{12}$ has columns $c_{1},...,c_{r}$ and $M_{21}$ has columns
$d_{1},....,d_{w}$. Define $\overline{f_{a_{1},...,a_{s_{1}+k_{1}-1},d_{1},...,d_{w}}^{l_{1},k_{1}}}$
and $\overline{f_{b_{1},...,b_{s_{1}+k_{2}-1,c_{1},...,c_{r}}}^{l_{2},k_{2}}}$
as in the Definition \ref{TopM}. Let $\{c_{1},...,c_{r},b_{1},...,b_{s_{1}+k_{2}-1}\}=\{d_{1},...,d_{w},a_{1},...,a_{s_{1}+k_{1}-1}\}=\mathcal{I}$,
from the way we define $M_{12}$ and $M_{21}$ we have the initial
term of $\overline{f_{\mathcal{I}}^{l_{1},k_{1}}}$ is $\mbox{in}(M_{12}f_{a_{1},...,a_{s_{1}+k_{1}-1}}^{l_{1},k_{1}})$
and similarly for $\overline{f_{\mathcal{I}}^{l_{2},k_{2}}}$. We
will like to apply Lemma \ref{MandM} to this case.

Rewrite $\overline{f_{\mathcal{I}}^{l_{2},k_{2}}}$ as $\alpha_{1}$:\begin{eqnarray*}
\alpha_{1}:= & {\displaystyle \sum_{r=k_{2}}^{s_{2}}(}-1)^{r+1}\left|\left[\begin{array}{c}
\begin{array}{c}
\overline{M_{12}}'\\
Y^{k_{2},k_{2}}\end{array}\\
Z^{l_{2},k_{2}-1}\\
Z^{r,r}\\
X^{1,l_{2}-1}\\
Y^{1,r-1}\\
Y^{r+1}\end{array}\right]_{\mathcal{I}}\right|+{\displaystyle \sum_{r=k_{2}}^{s_{2}}(-1)^{r+1}\sum_{u=r+1}^{s_{1}}}\left|\left[\begin{array}{c}
\begin{array}{c}
\overline{M_{12}}'\\
Y^{k_{2},k_{2}}\end{array}\\
Z^{l_{2},k_{2}-1}\\
X^{r,r}-Y^{r,r}\\
X^{1,l_{2}-1}\\
Y^{1,r-1}\\
Y^{r+1,u-1}\\
Z^{u,u}\\
X^{u+1,s_{1}}\end{array}\right]_{\mathcal{I}}\right|.\end{eqnarray*}
Notice that in the first sum of $\alpha_{1}$, when $r>k_{2}$, the
matrices have repeated row $y_{k_{2}}$. Hence the determinants are
zero. The first sum becomes $\alpha_{11}$: \[
\alpha_{11}:=\left|\left[\begin{array}{c}
\begin{array}{c}
\overline{M_{12}}'\\
Y^{k_{2},k_{2}}\end{array}\\
Z^{l_{2},k_{2}-1}\\
Z^{k_{2},k_{2}}\\
X^{1,l_{2}-1}\\
Y^{1,k_{2}-1}\\
Y^{k_{2}+1}\end{array}\right]_{\mathcal{I}}\right|\]
 We notice the leading term of $\alpha_{11}$ is $m_{12}(\mbox{in}f_{b_{1},...,b_{s_{1}+k-1}}^{l_{2}k_{2}})$.
Let the second sum of $\alpha_{1}$ be $\alpha_{12}$:\\
\[
\alpha_{12}:=\sum_{r=k_{2}}^{s_{2}}(-1)^{r+1}\sum_{u=r+1}^{s_{1}}\left|\left[\begin{array}{c}
\begin{array}{c}
\overline{M_{12}}'\\
Y^{k_{2},k_{2}}\end{array}\\
Z^{l_{2},k_{2}-1}\\
X^{r,r}-Y^{r,r}\\
X^{1,l_{2}-1}\\
Y^{1,r-1}\\
Y^{r+1,u-1}\\
Z^{u,u}\\
X^{u+1,s_{1}}\end{array}\right]_{\mathcal{I}}\right|.\]

Lemma \ref{MandM} provided if $\alpha_{12}$ is a combination of
elements of $\mathcal{G}$ such that the leading term of each summand
is smaller than $m_{12}f_{b_{1},...,b_{s_{1}+k_{2}-1}}^{l_{2}k_{2}}$.
Observe that in the sum of $\alpha_{12}$, when $r>k_{2}$, the matrices
have repeated row $y_{k_{2}}$. Hence their determinants are zero. 

We are only left with $r=k_{2}$, and $\alpha_{12}$ becomes 

\begin{eqnarray*}
\alpha_{13}:= & (-1{\displaystyle )^{k_{2}+1}\sum_{u=k_{2}+1}^{s_{1}}\left|\left[\begin{array}{c}
\begin{array}{c}
\begin{array}{c}
\overline{M_{12}}'\\
Y^{k_{2},k_{2}}\end{array}\end{array}\\
Z^{l_{2},k_{2}-1}\\
X^{k_{2},k_{2}}-Y^{k_{2},k_{2}}\\
X^{1,l_{2}-1}\\
Y^{1,k_{2}-1}\\
Y^{k_{2}+1,u-1}\\
Z^{u,u}\\
X^{u+1,s_{1}}\end{array}\right]_{\mathcal{I}}\right|}.\end{eqnarray*}
We apply Lemma \ref{Swithchg_ij} on $\alpha_{13}$, then $\alpha_{13}$
becomes $\alpha_{14}$:\begin{eqnarray*}
\alpha_{14}:= & (-1)^{k_{2}+1}{\displaystyle \sum_{u=k_{2}+1}^{s_{1}}\left|\left[\begin{array}{c}
\begin{array}{c}
\overline{M_{12}}'\\
Y^{k_{2},k_{2}}\end{array}\\
X^{l_{2},l_{2}}-Y^{l_{2},l_{2}}\\
Z^{l_{2}+1,k_{2}-1}\\
Z^{k_{2},k_{2}}\\
X^{1,l_{2}-1}\\
Y^{1,k_{2}-1}\\
Y^{k_{2}+1,u-1}\\
Z^{u,u}\\
X^{u+1,s_{1}}\end{array}\right]_{\mathcal{I}}\right|+(-1)^{k_{2}+1}\sum_{u=k_{2}+1}^{s_{1}}}\end{eqnarray*}
\begin{eqnarray*}
 &  & \sum_{\{p_{1},p_{2},q_{1},...,q_{s_{1}+k_{2}-1}\}=\mathcal{I}}\left(\pm(g_{l_{2}p_{1},k_{2},p_{2}}-g_{l_{2}p_{2},k_{2}p_{1}})\left|\left[\begin{array}{c}
\begin{array}{c}
\begin{array}{c}
\overline{M_{12}}'\\
Y^{k_{2},k_{2}}\end{array}\end{array}\\
Z^{l_{2}+1,k_{2}-1}\\
X^{1,l_{2}-1}\\
Y^{1,k_{2}-1}\\
Y^{k_{2}+1,u-1}\\
Z^{u,u}\\
X^{u+1,s_{1}}\end{array}\right]_{q_{1},...,q_{s_{1}+k-2}}\right|\right).\end{eqnarray*}
After removing the repeated row $y_{l_{2}}$ in the first sum in the
above expression for $\alpha_{14}$, let this sum be $\alpha_{15}$:\\
 \begin{eqnarray*}
\alpha_{15}: & =(-1)^{k_{2}+1}{\displaystyle \sum_{u=k_{2}+1}^{s_{1}}\left|\left[\begin{array}{c}
\begin{array}{c}
\begin{array}{c}
\overline{M_{12}}'\\
Y^{k_{2},k_{2}}\end{array}\end{array}\\
X^{l_{2},l_{2}}\\
Z^{l_{2}+1,k_{2}-1}\\
Z^{k_{2},k_{2}}\\
X^{1,l_{2}-1}\\
Y^{1,k_{2}-1}\\
Y^{k_{2}+1,u-1}\\
Z^{u,u}\\
X^{u+1,s_{1}}\end{array}\right]_{\mathcal{I}}\right|} & =\pm\sum_{u=k_{2}+1}^{s_{1}}(-1)^{u+1}\left|\left[\begin{array}{c}
\begin{array}{c}
\overline{M_{12}}'\end{array}\\
Z^{l_{2}+1,k_{2}}\\
Z^{u,u}\\
X^{1,l_{2}}\\
Y^{1,u-1}\\
X^{u+1,s_{1}}\end{array}\right]_{\mathcal{I}}\right|\end{eqnarray*}
Now $\alpha_{15}$ becomes\begin{eqnarray*}
\alpha_{16}:= & {\displaystyle \sum_{\{p_{1},...,p_{v-1},q_{1},...,q_{s_{1+k_{2}}}\}=\mathcal{I}}\pm|M_{p_{1},...,p_{v-1}}^{12}|}\left({\displaystyle \sum_{u=k_{2}+1}^{s_{1}}}(-1)^{u+1}\left|\left[\begin{array}{c}
Z^{l_{2}+1,k_{2}}\\
Z^{u,u}\\
X^{1,l_{2}}\\
Y^{1,u-1}\\
X^{u+1,s_{1}}\end{array}\right]_{\mathcal{I}}\right|\right)\\
= & {\displaystyle \sum_{\{p_{1},...,p_{v-1},q_{1},...,q_{s_{1+k_{2}}}\}=\mathcal{I}}}\pm|M_{p_{1},...,p_{v-1}}^{12}|p_{q_{1},...,q_{s_{1}+k_{2}}}^{l_{2}+1,k_{2}+1}.\end{eqnarray*}
Here $\{p_{q_{1},...,q_{s_{1}+k_{1}}}^{l_{2}+1,k_{2}+1}\}$ are as
defined in lemma \ref{f}, and the proof of lemma \ref{f} shows that
they are in $\mathcal{L}$ . This shows $\alpha_{12}$ is a combination
of elements of $\mathcal{G}$ such that the leading term of each summand
is smaller than $m_{12}f_{b_{1},...,b_{s_{1}+k_{2}-1}}^{l_{2}k_{2}}$. 

We can do the same to $\overline{f_{\mathcal{I}}^{l_{1},k_{1}}}$
and show the second part of the sum of $\overline{f_{\mathcal{I}}^{l_{1},k_{1}}}$
is a combination of elements of $\mathcal{G}$ such that the leading
term of each summand is smaller than in$(m_{21}f_{a_{1},...,a_{s_{1}+k_{1}-1}}^{l_{1}k_{1}})$. 

In case (b): assume $k_{1}=k_{2}$ and $l_{1}<l_{2}\leq k_{2}=k_{1}$.
The proof technique is very similar to case (a). Notice that the first
matrix appearing in the expression for $f_{a_{1},...,a_{s_{1}+k_{1}-1}}^{l_{1}k_{1}}$
has row $z_{l_{1}}$ without row $x_{l_{1}}$ and the first matrix
appearing in the expression for $f_{b_{1},...,b_{s_{1}+k_{1}-1}}^{l_{2}k_{2}}$
has row $x_{l_{1}}$ without row $z_{l_{1}}$. Since $l_{1}\leq l_{2}-1$
and $l_{1}\leq k_{2}-1$, each matrix of $\overline{}$$\overline{f_{b_{1},...,b_{s_{1}+k_{2}-1},c_{1},...,c_{r}}^{l_{2},k_{2}}}$
has the rows $x_{l_{1}}$ and $y_{l1}$. They also all have row $z_{l_{1}}$.
Applying lemma \ref{g_ij} gives all the determinants of those matrices
are in $(\{g_{l_{1}i,l_{1}j}\})$. \end{proof}
\begin{lem}
$h_{PQ}=0$ if $P$, $Q\in\{$$f_{a_{1},...,a_{s_{1}+k-1}}^{l,k}$,
$U_{p_{1},q_{1},a_{1},...,a_{s_{1}}}$ , \\
$W_{p,q,a_{1},...,a_{p},a_{s_{2}+1},...,a_{s_{1}},b_{1},...,b_{s_{2}}}$,
\\
$W_{p_{1},q_{1},a_{1},...,a_{p1},a_{s_{2}+1},...,a_{s_{1}},b_{1},...,b_{p_{1}+1},b_{p_{1}+2},b_{p_{1}+3}.,b_{s_{2}},b_{p_{1}+1}^{'},b_{p_{1}+2}^{'},...,b_{v}^{'}}^{p_{1}+1,v}$,
\\
$V_{p_{l},...,p_{k},a_{1},...,a_{k-1},b_{1},...,b_{s_{1}}}$, $V_{p_{l},...,p_{k},a_{1},...,a_{k-1},b_{1},...,b_{s_{1}},b_{k}^{'},b_{k+1}^{'},...,b_{w}^{'}}^{k,w}$,
$H_{a_{1},...,a_{s_{1}+k-1}}^{l,k,q}$ ,\\
 \textup{$I_{a_{s_{1}+k-1},...,a_{l+s_{2}-1},a_{l-1},...,a_{1},b_{1},...,b_{s_{2}}}^{l,k,q}$,}
$^{k,w}I_{q_{l},...,q_{k},b_{s_{1}},...,b_{s_{2}},...,b_{1},a_{l-1},...,a_{1},b_{k}^{'},...,b_{w}^{'}}^{l,k,q}\}$. \end{lem}
\begin{proof}
For the purpose of the proof, we drop the column indices. The remainders
of $S$-pair of $f^{l,k}$ and $U$ are in the ideal generated by
$(\{V\},\{$$W$\},$\{g\}$,\{$Y$\}). Similarly, the remainders of
$S$-pair of $f^{l,k}$ and $ $$W$ are in $(\{V^{k,w}$\}, \{$W^{p_{1}+1,v}$\},$\{g\}$,\{$Y$\}),
and the remainders of $S$-pair of $f^{l,k}$ and $W^{p_{1}+1,v}$
are in (\{$V^{k,v+1}\}$, $\{W^{p_{1}+1,v+1}\},\{g\},\{Y\})$. The
remainders of $S$-pair of $f^{l,k}$ and $V$ are $V^{k,w}$ and
the remainders of $S$-pair of $f^{l,k}$ and $V^{k,w}$ are $V^{k,w+1}$.
The remainders of $S$-pair of $f^{l,k}$ and $H^{l,k,q}$ are in
$(\{I^{l,k,q}\},\{g\},\{Y\})$ and the remainders of $S$-pair of
$f^{l,k}$ and $I^{l,k,q}$ are in $(\{{}^{k,w}I^{l,k,q}\}$, $\{g\}$,
$\{Y\})$. Finally the remainders of $S$-pair of $f^{l,k}$ and${}^{k,w}I^{l,k,q}$
are in (\{$^{k,w+1}I^{l,k,q}\}$, $\{g\},\{Y\})$. All the other $S$-pair
of elements have similar relationship as above.
\end{proof}
We complete the proof of Theorem \ref{GB}. 

\medskip{}

\textit{Proof of Lemma \ref{NZD}: }Notice that $x_{11}$ is a non
zero-divisor on $k[X,Y,Z]/\mbox{in}(\mathcal{L})$. Since the only
possible elements of $\mathcal{G}$ that leading monomials are divisible
by $x_{11}$ are $ $$U_{1,1,1,a_{2},...,a_{s_{1}}}$, $W_{1,1,1,a_{s_{2}+1},...,a_{s_{1}},b_{1},...,b_{s_{2}}}$
and \\
$W_{1,1,1,a_{s_{2}+1},...,a_{s_{1}},b_{1},...,b_{s_{2}},b_{2}^{'},...,b_{v}^{'}}^{2,v}$.
But those monomials are divisible by $ $the leading monomials of
$f_{1,a_{2},...,a_{s_{1}}}^{1,1}$, $V_{1,b_{1},...,b_{s_{2}},a_{s_{2}+1},...,a_{s_{1}}}$
and $V_{1,b_{1},...,b_{s_{2}},a_{s_{2}+1},...,a_{s_{1}},b_{2}^{'},...,b_{v}^{'}}^{2,v}$.
Hence $x_{11}$ is also a non zero-divisor on $k[X,Y,Z]/\mathcal{L}$.\hfill{}$\square$

\medskip{}

The following example is computed in Singular \cite{GPS}. This example
gives us an idea what does the initial ideal looks like.
\begin{example}
Let $X$, $Y$, $Z$ be a $3\times4$ matrices, $X_{3,4}$, $Y_{2,4}$
are $3\times4$ and $2\times4$ submatrices of $X$ and $Y$ then
the defining ideal of the $\mathcal{R}(\mathbb{D})$ is generated
by $I_{3}(X_{3,4})$, $I_{2}(Y_{2,4})$, $g_{ij,lk}$ where $1\leq i,l\leq3$,
$1\leq l,k\leq4$ and \begin{eqnarray*}
f_{a_{1},...,a_{3}} & = & \left|\begin{array}{ccc}
z_{1a_{1}} & z_{1a_{2}} & z_{1a_{3}}\\
x_{2a_{1}} & x_{2a_{2}} & x_{2a_{3}}\\
x_{3a_{1}} & x_{3a_{2}} & x_{3a_{3}}\end{array}\right|+\left|\begin{array}{ccc}
y_{1a_{1}} & y_{1a_{2}} & y_{1a_{3}}\\
z_{2a_{1}} & z_{2a_{2}} & z_{2a_{3}}\\
x_{3a_{1}} & x_{3a_{2}} & x_{3a_{3}}\end{array}\right|,\end{eqnarray*}
 where $1\leq a_{1}<a_{2}<a_{3}\leq4$. The initial ideal of $\mathcal{L}$
via the term order defined in Theorem \ref{GB} is generated by $\{x_{1a_{3}}x_{2a_{2}}x_{3a_{1}}\}_{1\leq a_{1}<a_{2}<a_{3}\leq4}$,
$\{y_{1b_{2}}y_{2b_{1}}\}_{1\leq b_{1}<b_{2}\leq4}$, $\{z_{ij}x_{lk}\}_{i<l\mbox{ or }i=l\mbox{ and \ensuremath{j<k}}}$,
$\{z_{1a_{1}}y_{2a_{2}}y_{3a_{3}}\}_{1\leq a_{1}<a_{3}<a_{2}\leq4}$,
$z_{11}z_{22}y_{34}y_{33}$, $z_{21}x_{14}y_{13}y_{32}$, $z_{12}z_{21}x_{12}y_{14}y_{33}$,
$z_{13}z_{21}x_{13}y_{14}y_{32}$, $z_{31}x_{14}x_{23}y_{32}$, $\{z_{2j}x_{1a_{3}}x_{2a_{2}}y_{3a_{1}}\}_{1\leq a_{1}<a_{2}\leq j<a_{3}\leq4,}$
$_{\mbox{ or }1\leq a_{2}\leq j<a_{1}<a_{3}\leq4}$, $\{z_{2j}x_{1a_{3}}y_{2a_{2}}y_{1a_{1}}\}_{1\leq a_{1}<a_{2}<a_{3},a_{2}<j}$,
$\{z_{1j}x_{1a_{3}}y_{2a_{2}}y_{3a_{1}}\}_{1\leq a_{1}<a_{2}<a_{3}\leq j\leq4,}$
$_{\mbox{ or }1\leq a_{1}<a_{3}\leq j<a_{2}\leq4}$, $\{z_{1j}x_{1a_{3}}y_{1b_{1}}y_{2b_{2}}y_{3b_{3}}\}_{1\leq b_{2}<b_{3}<a_{3}\leq j\leq4,}$$_{\mbox{ or }1\leq a_{1}<a_{3}\leq j<a_{2}\leq4}$.
We can see the variable $x_{11}$ is not in the generating set of
the initial ideal of $\mathcal{L}$.
\end{example}
\medskip{}

\curraddr{\begin{center}
{\small Department of Mathematics, }
\par\end{center}}

\curraddr{\begin{center}
{\small University of California, Riverside, CA 92521, USA}
\par\end{center}}

\email{\begin{center}
e-mail: linkuei@ucr.edu
\par\end{center}}
\end{document}